\documentclass[article]{siamart190516}
\usepackage{amsmath,amssymb}
\usepackage{caption}
\usepackage{hyperref}
\usepackage[utf8]{inputenc}
\usepackage{mathtools}
\usepackage{color,graphicx}
\usepackage{thmtools, thm-restate}
\usepackage{enumitem}
\reversemarginpar

\usepackage{algorithm}
\usepackage{algpseudocode}

\usepackage{tikz}
\usetikzlibrary{patterns}
\usetikzlibrary{arrows}
\usetikzlibrary{3d,calc}

\usepackage[onehalfspacing]{setspace}
\usepackage[colorinlistoftodos,textsize=tiny,textwidth=31mm]{todonotes}
\usepackage{etoolbox}
\usepackage{mathabx}
\makeatletter

\patchcmd{\@addmarginpar}{\ifodd\c@page}{\ifodd\c@page\@tempcnta\m@ne}{}{}
\providecommand*{\cupdot}{%
  \mathbin{%
    \mathpalette\@cupdot{}%
  }%
}
\newcommand*{\@cupdot}[2]{%
  \ooalign{%
    $\m@th#1\cup$\cr
    \hidewidth$\m@th#1\cdot$\hidewidth
  }%
}
\providecommand*{\bigcupdot}{%
  \mathop{%
    \vphantom{\bigcup}%
    \mathpalette\@bigcupdot{}%
  }%
}
\newcommand*{\@bigcupdot}[2]{%
  \ooalign{%
    $\m@th#1\bigcup$\cr
    \sbox0{$#1\bigcup$}%
    \dimen@=\ht0 %
    \advance\dimen@ by -\dp0 %
    \sbox0{\scalebox{2}{$\m@th#1\cdot$}}%
    \advance\dimen@ by -\ht0 %
    \dimen@=.5\dimen@
    \hidewidth\raise\dimen@\box0\hidewidth
  }%
}
\makeatother
\reversemarginpar
\newsiamremark{expl}{Example}

\newcommand{\cG}{\mathcal{G}}

\newcommand{\cF}{\mathcal{F}}

\newcommand{\cO}{\mathcal{O}}

\newcommand{\bitm}{\begin{itemize}}
\newcommand{\eitm}{\end{itemize}}
\newcommand{\bitme}{\begin{enumerate}[label=(\roman*),leftmargin=0.25in]}
\newcommand{\eitme}{\end{enumerate}}
\newcommand{\beq}{\begin{equation}}
\newcommand{\eeq}{\end{equation}}
\newcommand{\btcb}{\begin{tcolorbox}}
\newcommand{\etcb}{\end{tcolorbox}}
\def\bals#1\eals{\begin{align*} #1 \end{align*}}
\def\bal#1\eal{\begin{align} #1 \end{align}}

\newcommand{\with}{\quad \text{ with }  }

\newcommand{\where}{\quad \text{ where } }
\newcommand{\Ffor}{\quad \text{ for }}
\newcommand{\Aand}{\quad \text{ and } \quad }

\newcommand\Dom\Omega

\newcommand\RR{\mathbb{R}}

\newcommand\TT{\mathbb{T}}

\newcommand\NN{\mathbb{N}}

\newcommand\ZZ{\mathbb{Z}}

\newcommand\Lap\Delta

\newcommand\abs[1]{\left\lvert #1 \right\rvert}

\def\bpde#1\epde{\[\left\{\begin{aligned}#1\end{aligned}\right. \]}
\def\inbpde#1\inepde{\left\{\begin{aligned}#1\end{aligned}\right.}
\def\binpde#1\einpde{\left\{\begin{aligned}#1\end{aligned}\right.}

\newcommand\brkt[2]{\left\langle {#1},{#2} \right\rangle}

\newcommand\Norm[2]{\lVert { #1 } \rVert_{#2}}

\def\cC{\mathcal{C}}

\def\half{\frac{1}{2}}

\def\bfA{\mathbf{A}}

\def\b0{\mathbf{0}}

\def\eps{\varepsilon}

\def\bbmat{\begin{bmatrix}[r]}
\def\ebmat{\end{bmatrix}}

\newcommand{\barr}{\begin{array}}
\newcommand{\ea}{\end{array}}
\newcommand{\bea}{\begin{eqnarray}}
\newcommand{\eea}{\end{eqnarray}}
\newcommand{\bt}{\begin{table}}
\newcommand{\et}{\end{table}}

\DeclareMathOperator\Id{Id}
\DeclareMathOperator\Span{span}

\DeclareMathOperator\supp{supp}

\theoremstyle{plain}

\newsiamthm{cond}{Condition}
\newsiamremark{remark}{Remark} 
\newsiamremark{example}{Example}

\numberwithin{equation}{section}

\newcommand\tturl[1]{{\tt \scriptsize [\url{{#1}}]}}

\newcommand{\TheTitle}{\textsf{A Range characterization of the single-quadrant ADRT}}

\newcommand{\oI}{I}

\newcommand\bitand{\mathbin{\land}}
\newcommand\bitnot{\mathbin{\neg}}
\newcommand\bitxor{\mathbin{\underline{\lor}}}
\newcommand\bitsum{\mathbin{\varsigma_2}}
\newcommand\bitiprod{\mathbin{*}}
\newcommand\brem{\mathbin{\%}}      

\newcommand\uPhi{\smash{\underline{\Phi}}}
\newcommand\uPsi{\smash{\underline{\Psi}}}

\newcommand\oPhi{\smash{\overline{\Phi}}}
\newcommand\oPsi{\smash{\overline{\Psi}}}

\newcommand\sqbrkt[2]{\left[ {#1},\,{#2} \right]}
\newcommand\sqbrkte[3]{\left[ {#1},\,{#2},\,{#3} \right]}

\newcommand\nm[1]{{#1}^n_{m}}
\newcommand\nmz[2]{{#1}^n_{#2}}
\newcommand\nml[2]{{#1}^n_{m,#2}}
\newcommand\nmlz[3]{{#1}^n_{#2,#3}}
\newcommand\nmlo[1]{{#1}^n_{m,\ell}}
\newcommand\nmo[1]{{#1}^n_m}

\newcommand\nmlq[1]{{#1}^n_{m,\ell,q}}

\newcommand\nmlpqz[3]{{#1}^n_{m,\ell,#2,#3}}

\newcommand\inmzS[1]{(\nmz{S}{#1})^{-1}}

\newcommand\indic{\chi}

\begin{document}

\ifpdf
\DeclareGraphicsExtensions{.pdf, .jpg, .tif}
\else
\DeclareGraphicsExtensions{.eps, .jpg}
\fi

\title{\TheTitle}

\author{%
  Weilin Li\thanks{Courant Institute of Mathematical Sciences, %
  New York University, New York, NY 10012 %
  ({\tt weilinli@cims.nyu.edu})}%
  \and 
  Kui Ren%
  \thanks{Department of Applied Physics and Applied Mathematics, Columbia
  University, New York, NY 10027 ({\tt kr2002@columbia.edu})}
  \and
  Donsub Rim%
  \thanks{Department of Mathematics and Statistics, Washington University in St. Louis, St. Louis, MO 63105 ({\tt rim@wustl.edu})}
}
\maketitle

\begin{abstract} 
This work characterizes the range of the single-quadrant approximate discrete
Radon transform (ADRT) of square images. The characterization follows from a set
of linear constraints on the codomain. We show that for data satisfying these
constraints, the exact and fast inversion formula [Rim, \emph{Appl. Math.
Lett.} \textbf{102} 106159, 2020] yields a square image in a stable manner. The
range characterization is obtained by first showing that the ADRT is a
bijection between images supported on infinite half-strips, then identifying the
linear subspaces that stay finitely supported under the inversion formula. 
\end{abstract}


\begin{AMS}
\texttt{44A12}, \texttt{65R10}, \texttt{92C55}, \texttt{68U05}, \texttt{15A04}
\end{AMS}

\section{Introduction}

The Radon transform of functions defined on the Euclidean plane $\RR^2$
is a transform that computes the integral of the function over all straight
lines \cite{Radon1917}. The lines are parametrized by two variables, so the
transformed function is defined on the two dimensional (2D) cylinder $\RR \times
[-\pi, \pi]$. The Radon transform plays a fundamental role in many areas in both
theoretical and applied mathematics.  Perhaps the most widely known application
is in the field of inverse problems, as a mathematical model of medical imaging
techniques like computerized tomography (CT) \cite{Cormack63, Natterer}. We will
refer to the Radon transform as the \emph{continuous} Radon transform, to
distinguish it from its discretizations. 

The range of continuous Radon transform, when it is applied to the class
of smooth compactly supported functions on $\RR^2$, form a strict linear
subspace of the codomain of smooth compactly supported functions on the cylinder
\cite{Helgason99}.  Naturally, the inverse of the transform is well-defined only
on the range, so it is important to be able to verify whether the given data to
be inverted lies in the range. Range characterization refers to a precise
description of this strict subspace. In most applications the data almost never
lies in the range due to measurement noise, and in such situations the range
characterization enables one to derive accurate and efficient methods for
mapping the data back into the range so the inverse can be computed.  These
theoretical and practical issues have been studied carefully in the past, and
we refer the reader to standard texts on the topic for general review
\cite{Dean83, Helgason99, Natterer}.

There are many range characterization results for the continuous Radon
transform and its generalizations. In fact, the singular value decomposition
(SVD) of the transform is known in various settings \cite{Natterer,
Marr1974OnTR, Quinto83, Maass87, Maass92, KazantsevBukhgeim04}.  The
decomposition is a fundamental tool for analyzing and building algorithms in
situations where there are only finite or incomplete data \cite{Louis84,
Louis85, LouisRieder89, Louis96}. In case the function to be transformed is
defined on the torus $\mathbb{T}^n = \mathbb{R}^n / \mathbb{Z}^n$ there are
explicit results given in terms of the Fourier coefficients \cite{Ilmavirta2015,
Railo2020}, and for tensors on the slab $[0, 1] \times \TT^n$ a non-trivial
kernel can be characterized \cite{Ilmavirta2018}. Furthermore, range
characterizations in the Riemannian setting is currently an area of active
research \cite{IlmavirtaMonard2019}. 

To perform practical computations, it is necessary to consider finite
dimensional discretizations of the continuous Radon transform.  Many of these
discretizations were suitably devised for specific applications, but a
general-purpose numerical discretization of the transform has only more recently
attracted careful attention. Brief overviews of these methods appear in
\cite{Press06drt, Averbuch08drt}. These discretizations of the Radon transform
also typically map a discretized 2D function onto a range that forms a strict
subspace of the codomain, as in the continuous case. This work concerns the
range characterization of a particular choice of discretization.

The Approximate Discrete Radon Transform (ADRT) is a fast multi-resolution
algorithm that approximates the continuous Radon transform
\cite{Brady98adrt,Gotz96fdrt}. The ADRT substitutes the integral over straight
lines in the continuous Radon transform with a sum of point-values lying in the
so-called digital lines. The ADRT for a 2D image is assembled from four
identical single-quadrant ADRTs, performed on the properly rotated versions of
the image. 

The single-quadrant ADRT is a mapping from $N \times N$ square images with $N^2$
degrees of freedom to a larger codomain with $3 N^2/ 2 - N/2$ degrees of
freedom. In this work, we study this redundancy in the codomain and characterize
the range of ADRT. We will specify precise constraints on the codomain that
yields the range. We begin by considering the single-quadrant ADRT to be a map
taking the space of images supported on an infinite half-strip into itself and
then showing that it is a bijection in that setting, by the virtue of its exact
inversion formula. Next we show that, for the original image to have been
finitely supported, its ADRT must satisfy a set of linear constraints. We shall
show that there are $N(N-1)/2$ such linearly independent constraints, thereby
prescribing the range that necessarily has $N^2$ degrees of freedom.  The
range characterization presented in this work is the first of its kind for any
of the fast discretizations of the Radon transform, to the best of our
knowledge.

It was the work of Press \cite{Press06drt} that first demonstrated that the ADRT
can be inverted to numerical precision by devising a multi-grid method. Press
also conjectured the existence of an multi-grid variant of complexity $\cO(N^2
(\log N)^2 \log \eps)$ for given error threshold $\eps \in (0,1)$. Subsequently,
the single-quadrant ADRT was found to have an exact and fast inverse of cost
$\cO(N^2 \log N)$ in \cite{Rim20iadrt}. This further revealed the ADRT to have a
four-fold redundancy. The range characterization in this work is a
consequence of this inverse, and it can be potentially useful in developing
new projection techniques or iterative methods for inverting the ADRT.

Various other efficient discretizations of the Radon transform have been
proposed \cite{Beylkin87drt,Kelley93frt,Matus93finitert,
Hsung96periodicrt,Averbuch08drt,Averbuch08ppft,Ilmavirta20}, along with methods
for computing their inverses. What sets ADRT apart is that it has an
inversion formula that is both \emph{fast} and \emph{exact}; the inverse can be
computed in $\cO(N^2 \log N)$ operations, and in the absence of numerical error
in the finite-precision operations, the formula would recover the original image
exactly.  Furthermore, this inverse for the single-quadrant ADRT is derived
neither from the normal equations of the forward transform, nor the Fourier
slice theorem \cite{Natterer}. Rather, it exploits a type of localization
property, that the single-quadrant ADRT of a half-image can be computed from
that of the full-image in $\cO(N^2)$ operations.  Despite enjoying these
properties that the continuous Radon transform does not, the ADRT still
converges to the continuous transform as $N \to \infty$ in the case the square
image was formed by appropriately sampling a Lipschitz continuous image
\cite{Brady98adrt,Gotz96fdrt}.  

The efficiency of both the forward and the inverse ADRT makes it suitable for
use in applications. Besides from its well-known importance in tomography
problems arising in medical imaging, radar imaging, geophysical imaging
\cite{Dean83, Natterer} or electron microscopy \cite{Midgley2003, Frank1996}, it
is also useful in computing generalizations of Lax-Philips representation
\cite{Lax64scattering,Bonneel15slice,Rim18mr,Rim18split}. While the range
characterization is the sole focus here, the characterization will be
useful in these applications as well. 

\section{Digital lines and the ADRT}
We will introduce notations and definitions, and recall the single-quadrant ADRT
and its dual. Most of this section is a review of known
definitions and properties \cite{Brady98adrt, Gotz96fdrt,Press06drt}.

We denote the set of integers as $\ZZ$ and that of non-negative integers as $\NN$. We will make use of a set of indices $\oI_m := \{i \in \NN: i < m\}$ with $m \in \NN$. Let us define the binary operator $\brem: \NN \times \NN \to \NN$ as the remainder $a \brem b := a - \lfloor a/b \rfloor \cdot b$ for $a,b \in \NN.$

We define the linear spaces of images defined on the vertical strip $\ZZ \times I_{2^n}$,
\[
    \begin{aligned}
    \cF^n
    &:= 
    \left\{ 
        f: \ZZ \times \oI_{2^n} \to \RR
    \right\}
    &&
    \text{(images on the vertical strip)},
    \\
    \cF^n_+ 
    &:= 
    \left\{ 
        f \in \cF^n: 
           f(i,\cdot) \equiv 0 \text{ for } i \le b_f,\,
           b_f \in \ZZ
        \right\}
    &&
    \text{(images on the vertical half-strip)},
    \\
    \cF^n_0 
    &:= 
    \left\{ 
        f \in \cF^n: 
        \supp f \subset I_{2^n} \times I_{2^n}
        < \infty 
    \right\}
    &&
    \text{(images on the square)}.
    \end{aligned}
\]
As the naming suggests, $\cF^n$ is the space of images defined on the vertical strip $\ZZ \times I_{2^n}$, $\cF^n_+$ contains images whose support is bounded below, that is, its support is contained in a vertical half-strip, and $\cF^n_0$ is the space of square images, as extended to the vertical strip. Needless to say, $\cF^n_0 \subset \cF^n_+ \subset \cF^n$. 

\newpage

The dot product for $f,g \in \cF^n$ is given by $f\cdot g := \sum_{(i,j) \in \ZZ \times I_{2^n}} f(i,j) g(i,j)$. We will set the grid-size by $N := 2^n$ for some $n \in \NN$. 

A section is a restriction of image $f$ to a vertical strip of width $2^m$.

\begin{definition}\label{def:section}
Given $f \in \cF^n$, $m \in \oI_{n+1}$, $\ell \in \oI_{2^{n-m}}$, we define \emph{$(m,\ell)$-section of $f$} as $\nmlo{f} \in \cF^m$ taking on the values
\beq \label{eq:fnml}
    \nmlo{f}(i,j)
    :=
    f(i,j + \ell \, 2^m),
    \quad
    (i,j) \in \ZZ \times I_{2^m}.
\eeq
\end{definition} 
For each $f \in \cF^n$, the section $f^n_{m,\ell}$ merely accesses the entries that lie on the narrower vertical strip $\{(i,j+\ell 2^m) \,|\, (i,j) \in \ZZ \times I_{2^m} \}$. It is sometimes convenient to reverse the assignment \cref{eq:fnml}, 
\beq \label{eq:section-global}
    f(i,j) = \nml{f}{\lfloor j / 2^m \rfloor} (i, j \brem 2^m ),
    \quad
    (i,j) \in \ZZ \times I_{2^n}.
\eeq

Throughout, we will adopt the convention that whenever the superscript and the
first subscript agree, the subscripts are suppressed. \emph{For example,} $f^n:=
f^n_n := f^n_{n,0}$.

The ADRT approximates the continuous Radon transform by substituting an integral over the straight lines in favor of a sum over the so-called digital lines, which we now define. A digital line is a recursively defined collection of points in $\ZZ \times I_{2^n}$.

\begin{figure}
\centering
    \begin{tabular}{cc}
    \includegraphics[width=0.40\textwidth]{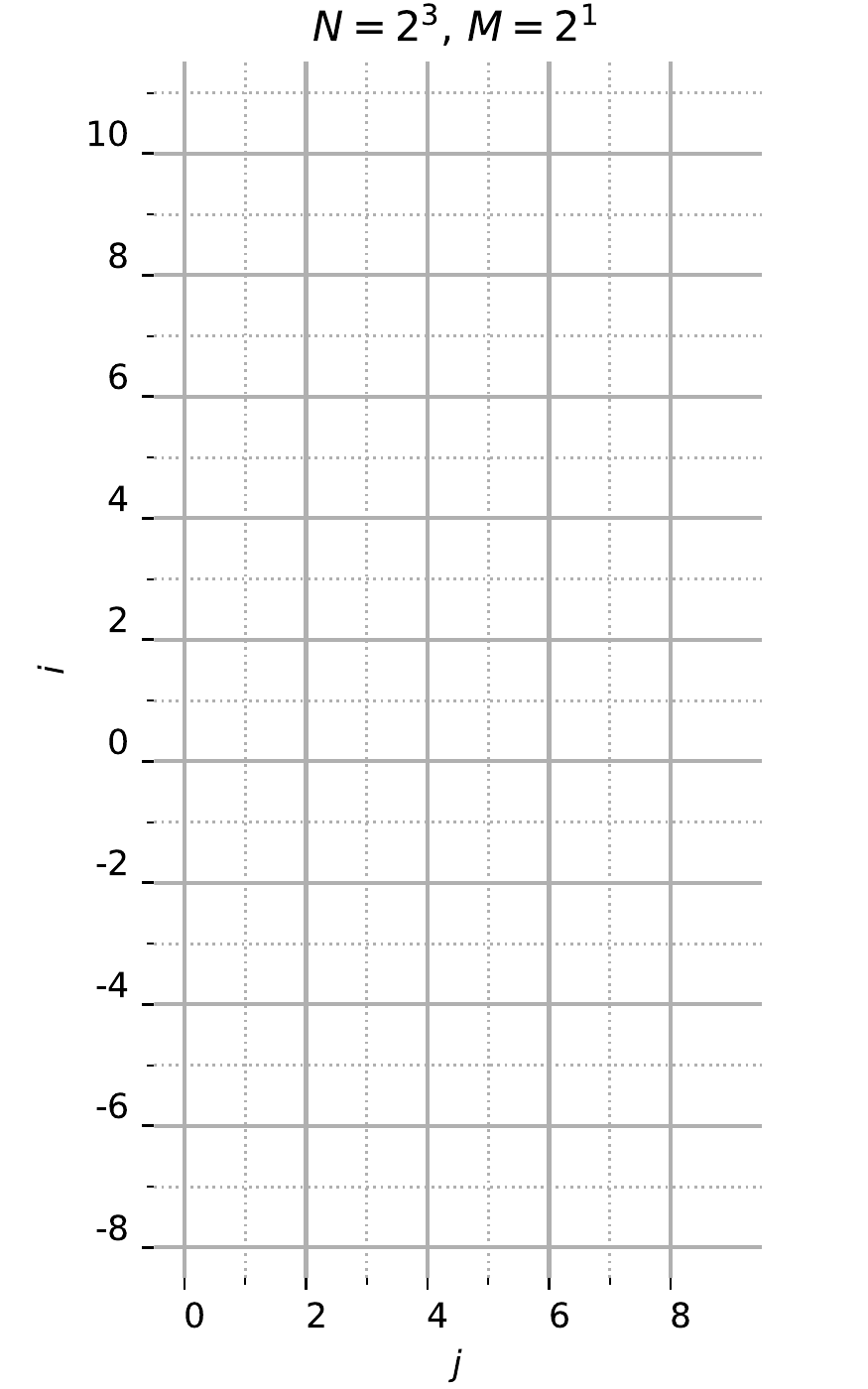}
    &
    \includegraphics[width=0.40\textwidth]{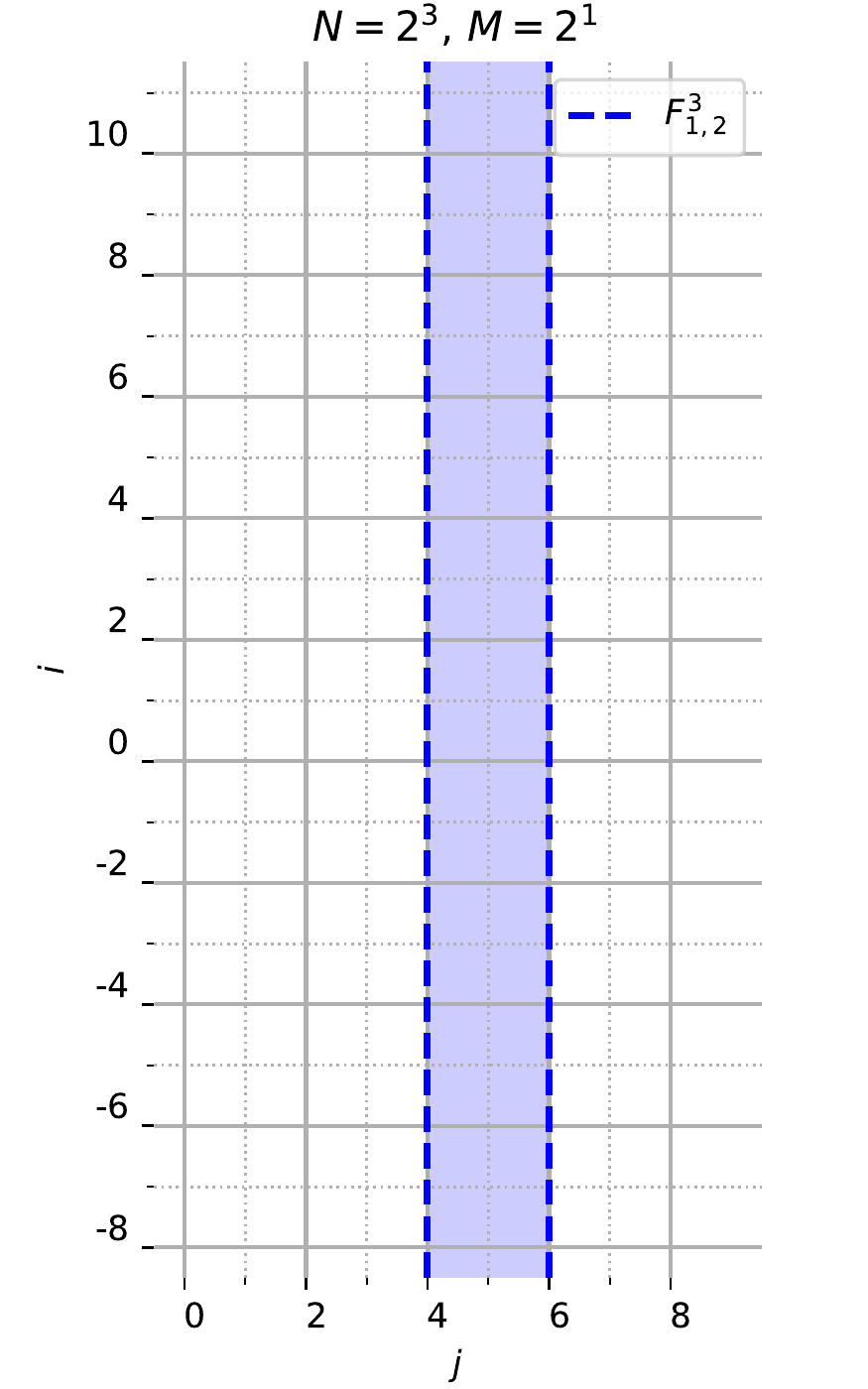}
    \end{tabular}
\caption{A plot of the grid $\ZZ \times \oI_{2^n}$ and its indexing (left), and
a plot highlighting the indices used for the section $f^n_{m,\ell}$
in \cref{def:section} (right) for the case $n=3, m=1, \ell = 2$.}
\end{figure}

\clearpage 

\begin{definition}\label{def:dline}
A \emph{digital line} $D^n_{m,\ell} \brkt{h}{s}$ for $(h,s) \in \ZZ
\times \oI_{2^m}$ is a subset of $\ZZ \times \oI_{2^n}$ that is defined
recursively. Letting $s = 2t$ or $s = 2t+1$,
\beq
    \left\{
    \begin{aligned}
    D^n_{m,\ell} \brkt{h}{2t+1}
      &:= D^n_{m-1, 2\ell} \brkt{h}{t} 
     \cup D^n_{m-1, 2\ell+1} \brkt{h+t+1}{t},\\
    D^n_{m,\ell} \brkt{h}{2t}
      &:= D^n_{m-1, 2\ell} \brkt{h}{t} 
      \cup D^n_{m-1, 2\ell+1} \brkt{h+t}{t},
    \end{aligned}
    \right.
    \label{eq:dline}
\eeq
and the relation is initialized by $D_{0,\ell}^n \brkt{h}{\ell} := \{(h,\ell)\}$ for $h \in \ZZ$, $\ell \in I_{2^n}$.
\end{definition} 
By its definition, the points of the digital line $D^n_{m,\ell}$ are contained in the narrow vertical strip $\ZZ \times \{s + \ell 2^m \,|\, s \in I_{2^m}\}$ of width $2^m$. The digital line can also be defined to take as arguments $(h,s) \in \ZZ \times I_{2^n}$, that is 
\beq
    D^n_m \brkt{h}{s} := D^n_{m, \lfloor s / 2^m \rfloor} \brkt{h}{s \brem 2^m},
    \quad
    (h,s) \in \ZZ \times I_{2^n},
\eeq
in a similar manner as \cref{eq:section-global}.

\begin{figure}
\centering
\def\svgwidth{0.95\textwidth}
{\scriptsize 
\begingroup%
  \makeatletter%
  \providecommand\color[2][]{%
    \errmessage{(Inkscape) Color is used for the text in Inkscape, but the package 'color.sty' is not loaded}%
    \renewcommand\color[2][]{}%
  }%
  \providecommand\transparent[1]{%
    \errmessage{(Inkscape) Transparency is used (non-zero) for the text in Inkscape, but the package 'transparent.sty' is not loaded}%
    \renewcommand\transparent[1]{}%
  }%
  \providecommand\rotatebox[2]{#2}%
  \newcommand*\fsize{\dimexpr\f@size pt\relax}%
  \newcommand*\lineheight[1]{\fontsize{\fsize}{#1\fsize}\selectfont}%
  \ifx\svgwidth\undefined%
    \setlength{\unitlength}{590.54635476bp}%
    \ifx\svgscale\undefined%
      \relax%
    \else%
      \setlength{\unitlength}{\unitlength * \real{\svgscale}}%
    \fi%
  \else%
    \setlength{\unitlength}{\svgwidth}%
  \fi%
  \global\let\svgwidth\undefined%
  \global\let\svgscale\undefined%
  \makeatother%
  \begin{picture}(1,0.52735431)%
    \lineheight{1}%
    \setlength\tabcolsep{0pt}%
    \put(0,0){\includegraphics[width=\unitlength,page=1]{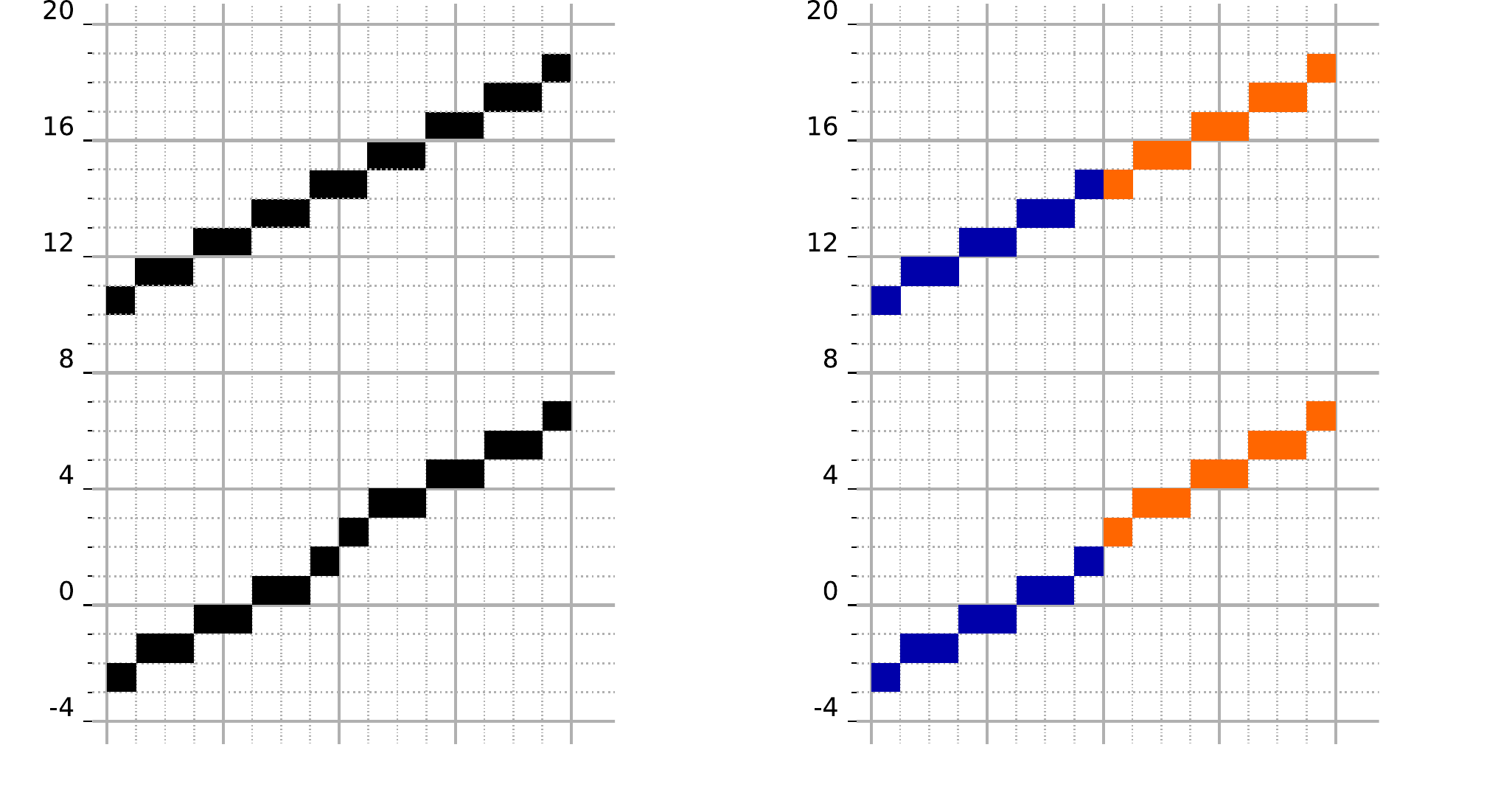}}%
    \put(0.59771805,0.17046433){\color[rgb]{0,0,0.66666667}\makebox(0,0)[lt]{\lineheight{1.25}\smash{\begin{tabular}[t]{l}$D^4_{3,0}\brkt{-3}{4}$\\\end{tabular}}}}%
    \put(0.74641877,0.15076583){\color[rgb]{1,0.4,0}\makebox(0,0)[lt]{\lineheight{1.25}\smash{\begin{tabular}[t]{l}$D^4_{3,1}\brkt{2}{4}$\\\end{tabular}}}}%
    \put(0.59771805,0.43970655){\color[rgb]{0,0,0.66666667}\makebox(0,0)[lt]{\lineheight{1.25}\smash{\begin{tabular}[t]{l}$D^4_{3,0}\brkt{10}{4}$\\\end{tabular}}}}%
    \put(0.74641877,0.37936771){\color[rgb]{1,0.4,0}\makebox(0,0)[lt]{\lineheight{1.25}\smash{\begin{tabular}[t]{l}$D^4_{3,1}\brkt{14}{4}$\\\end{tabular}}}}%
    \put(0.09733395,0.41938638){\color[rgb]{0,0,0}\makebox(0,0)[lt]{\lineheight{1.25}\smash{\begin{tabular}[t]{l}$D^4_{4,0}\brkt{10}{8}$\\\end{tabular}}}}%
    \put(0.09479393,0.18824453){\color[rgb]{0,0,0}\makebox(0,0)[lt]{\lineheight{1.25}\smash{\begin{tabular}[t]{l}$D^4_{4,0}\brkt{-3}{9}$\\\end{tabular}}}}%
    \put(0,0){\includegraphics[width=\unitlength,page=2]{dline-diagram.pdf}}%
  \end{picture}%
\endgroup%
}
\caption{A diagram depicting the recursive definition of a digital line (\cref{def:dline}).}
\end{figure}

We define the single-quadrant ADRT as a summation of point values of $f$ over the digital lines.

\begin{definition} \label{eq:adrt}
Given the image $f \in \cF^n$ and indices $m \in I_{n+1}$, $\ell \in I_{2^{n-m}}$, the \emph{$(m,\ell)$-single-quadrant ADRT} $R^n_{m,\ell}: \cF^n \to \cF^m$ is a linear operator given by
\beq \label{eq:adrt-sum}
    R^n_{m,\ell}[f](h,s) 
    := 
    \sum_{\mathclap{(i,j) \in D^n_{m,\ell} \brkt{h}{s} }} \, f(i,j),
    \quad
    (h,s) \in \ZZ \times I_{2^m}.
\eeq  
We also define the \emph{$m$-single-quadrant ADRT} by $R^n_m: \cF^n \to \cF^n$,
\beq
    R^n_m[f](h,s)
    :=
    \nml{R}{\lfloor s/2^m \rfloor} [f](h, s \brem 2^m)
    =
    \sum_{\mathclap{(i,j) \in D^n_m \brkt{h}{s} }} \, f(i,j),
    \quad
    (h,s) \in \ZZ \times I_{2^n}.
\eeq
In particular, when $m=n$ we call the sum the \emph{single-quadrant ADRT of $f$}, denoted by $R^n[f] := R^n_n[f] =  R^n_{n,0}[f]$.
\end{definition}
The two definitions $\nmlo{R}[f]$ and $\nmo{R}[f]$ are two different ways of indexing an identical set of values: If $g = \nmo{R}[f] \in \cF^n$ then its section $\nmlo{g} = \nmlo{R}[f] \in \cF^m$. Going forward, we will use the indices $(i,j) \in \ZZ \times I_{2^n}$ when referencing the individual values of an image $f$, and indices $(h,s) \in \ZZ \times I_{2^n}$ when referencing digital lines $\nmlo{D}\brkt{h}{s}$ or the ADRT $\nmlo{R}[f](h,s)$. The notation $(h,s)$ is due to Press \cite{Press06drt}; $h$ stands for height and $s$ for slope, respectively.

Observe that for $R^n_m$ the codomain is equal to the domain, both $\cF^n$. This corresponds to how the continuous Radon transform maps the vertical strip $\RR \times [0,1]$ into $\RR \times [0,\pi]$ (see \cite{Natterer}) except that the single-quadrant ADRT only concerns a quarter of the angles $[0,\pi/4]$ rather than $[0,\pi]$.

Due to the recursive definition of the digital lines in \cref{def:dline}, the
summation \cref{eq:adrt-sum} can also be computed recursively. Let us define the
linear operator $S_m: \cF^{m-1} \times \cF^{m-1} \to \cF^{m}$ by
\beq \label{eq:Sm}
    \left\{
    \begin{aligned}
        S_m[f_0,f_1](h, 2t) 
            &= f_0(h,t) + f_1(h+t,t),
            \\
        S_m[f_0,f_1](h, 2t+1) 
            &= f_0(h,t) + f_1(h+t+1,t).
    \end{aligned}
    \right.
\eeq
The same operation can be defined on $\cF^n$: Let us define $S^n_m: \cF^n \to \cF^n$ as given by
\beq
    S^n_m[f] := g,
    \where
    g^n_{m,\ell'} 
    = 
    S_m[f^n_{m-1,2\ell'}, f^n_{m-1,2\ell'+1}]
    \text{ for }
    \ell' \in I_{2^{n-m-1}}.
\eeq
The operator $S^n_m$ merely applies $S_m$ to the pairs of the individual
sections of an image in $\cF^n$, so it is well-defined and linear.  Put another
way, the operation by $S^n_m$ can be broken down into four smaller steps: 
\begin{enumerate}[label=(\roman*)]
    \item Take an image in $f \in \cF^n$ and extract its $(m-1, \ell)$-sections
    $f^n_{m-1, \ell} \in \cF^{m-1}$,
    \item pair the sections $(f^n_{m-1, 2\ell'}, f^n_{m-1, 2\ell' +1}) \in
    \cF^{m-1} \times \cF^{m-1}$ where $\ell' \in I_{2^{n-m-1}}$,
    \item apply $S_m$ to these $2^{n-m-1}$ pairs and obtain images in $\cF^m$, 
    \item create a new image $g \in \cF^n$ by defining its $(m, \ell')$-sections
    $\cF^m$ as the images in $\cF^m$ obtained in the previous step.
\end{enumerate}

Then $R^n_m[f]$ is equal to a sequential application of $S^n_m$
\beq
    R^n_m[f] 
    = 
    S^n_m \circ S^n_{m-1} 
    \circ \cdots \circ 
    S^n_1[f^n_{m,\ell}],
    \quad
    m \in I_{n+1}.
\eeq
If $f \in \cF^n_0$, then a single application of $S^n_m$ has the computational complexity $\cO(N^2)$, thus $R^n_m[f]$ costs $\cO(N^2 \log N)$.

\subsection{Dual digital lines and back-projection}

We will next describe the dual digital lines and the back-projection. An interesting feature of the dual of digital lines satisfy a similar recursive property; they are almost themselves digital lines, save that they have negative slopes.

\begin{definition} \label{def:dualdline}
A \emph{dual digital line} $D'^n_{m,\ell}\brkt{i}{j}$ for $(i,j) \in \ZZ \times I_{2^m}$ is defined as
\beq
    D'^n_{m,\ell} \brkt{i}{j}
    :=
    \left\{ 
       (h,s ) \in \ZZ \times I_{2^n} 
       \,|\,
       (i,j + \ell\,2^m) \cap D^n_m \brkt{h}{s} \ne \emptyset 
    \right\}.
\eeq
\end{definition} 
The dual digital lines $D^n_{m,\ell}$ are defined as the collection of digital lines that pass through the point $(i,j+\ell, 2^m) \in \ZZ \times I_{2^n}$. Since digital lines $D^n_m$ are indexed by $(h,s) \in \ZZ \times I_{2^n}$, this collection of digital lines can be identified to a collection of points in $\ZZ \times I_{2^n}$. In this sense, the dual digital lines $D'^n_{m,\ell}$ also form a collection of these points, like the digital lines. See \cref{fig:ddline} for an example. Furthermore, as was done for $D'^n_{m,\ell}$, we let
\beq
    D'^n_m \brkt{i}{j} 
    := 
    D'^n_{m,\lfloor j / 2^m \rfloor} \brkt{i}{j \brem 2^m},
    \quad
    (i,j) \in \ZZ \times I_{2^n}.
\eeq
Now, just as the sum over the digital lines define the ADRT $R^n_m$, the sum over the dual digital lines define the back-projection of the single-quadrant ADRT.
\begin{definition}
We define the $(m,\ell)$-\emph{back-projection} as $\nmlo{R'}: \cF^n \to \cF^m$ as a linear operator  given by
\beq
    \nmlo{R'} [g](i,j)
    :=
    \sum_{\mathclap{(h,s) \in D^{'n}_{m,\ell}\brkt{i}{j}}} \, g(h,s),
    \quad
    (i,j) \in \ZZ \times I_{2^m},
\eeq
and we will also define the $m$-\emph{back-projection} $R'^n_m: \cF^n \to \cF^n$ by
\beq
    \nmo{R'} [g](i,j)
    :=
    \nml{R'}{\lfloor j / 2^m \rfloor} [g] (h, j \brem 2^m)
    =
    \sum_{\mathclap{(h,s) \in D'^{n}_{m}\brkt{i}{j}}} \, g(h,s),
    \quad
    (i,j) \in \ZZ \times I_{2^n}.
\eeq
\end{definition}

\begin{figure}
\centering
\def\svgwidth{0.9\textwidth}
{\scriptsize 
\begingroup%
  \makeatletter%
  \providecommand\color[2][]{%
    \errmessage{(Inkscape) Color is used for the text in Inkscape, but the package 'color.sty' is not loaded}%
    \renewcommand\color[2][]{}%
  }%
  \providecommand\transparent[1]{%
    \errmessage{(Inkscape) Transparency is used (non-zero) for the text in Inkscape, but the package 'transparent.sty' is not loaded}%
    \renewcommand\transparent[1]{}%
  }%
  \providecommand\rotatebox[2]{#2}%
  \newcommand*\fsize{\dimexpr\f@size pt\relax}%
  \newcommand*\lineheight[1]{\fontsize{\fsize}{#1\fsize}\selectfont}%
  \ifx\svgwidth\undefined%
    \setlength{\unitlength}{714.00004255bp}%
    \ifx\svgscale\undefined%
      \relax%
    \else%
      \setlength{\unitlength}{\unitlength * \real{\svgscale}}%
    \fi%
  \else%
    \setlength{\unitlength}{\svgwidth}%
  \fi%
  \global\let\svgwidth\undefined%
  \global\let\svgscale\undefined%
  \makeatother%
  \begin{picture}(1,0.72605042)%
    \lineheight{1}%
    \setlength\tabcolsep{0pt}%
    \put(0,0){\includegraphics[width=\unitlength,page=1]{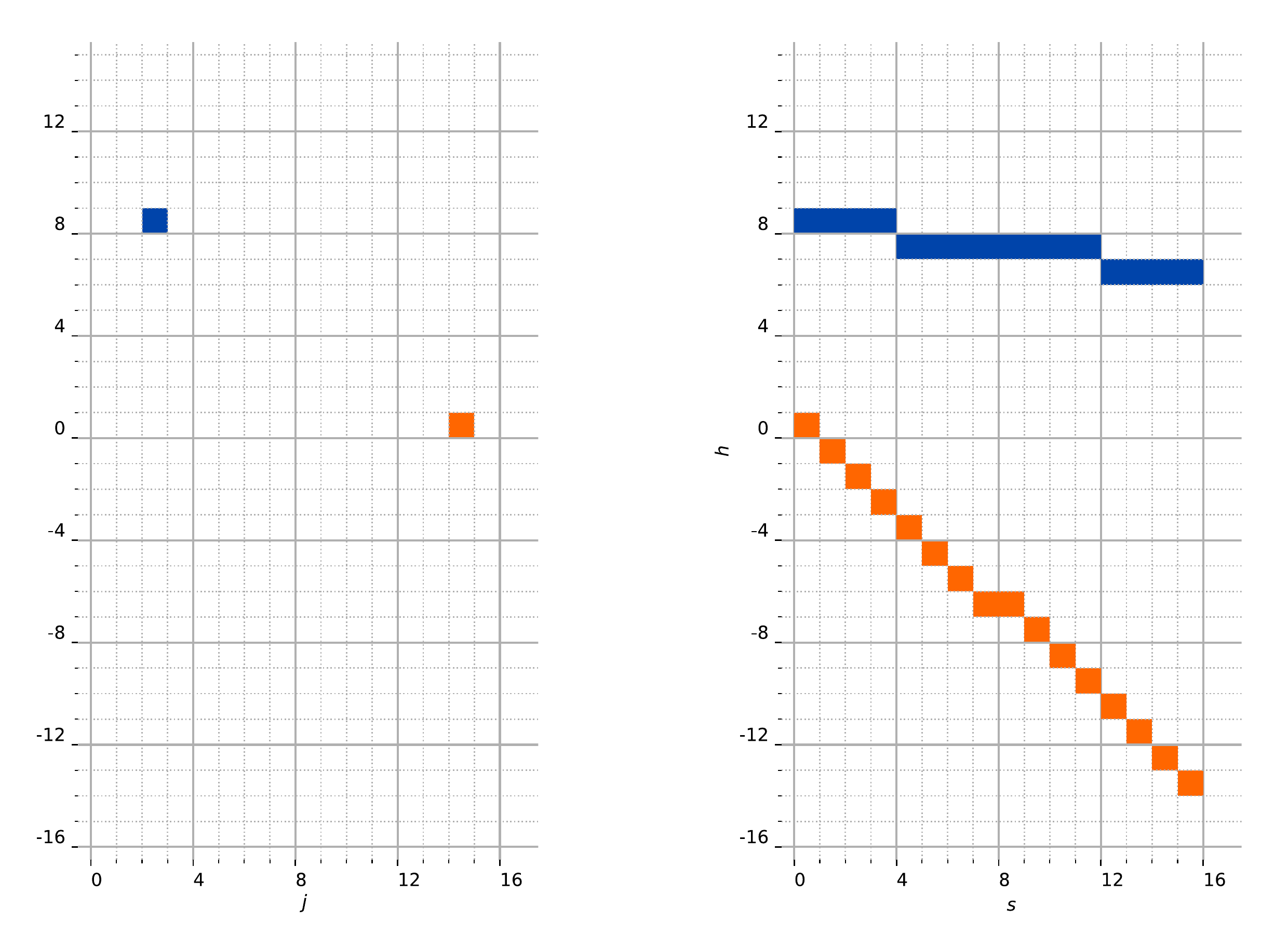}}%
    \put(0.71125689,0.57110839){\color[rgb]{0,0.26666667,0.66666667}\makebox(0,0)[lt]{\lineheight{1.25}\smash{\begin{tabular}[t]{l}$D'^4_{4,0}\brkt{8}{2}$\end{tabular}}}}%
    \put(0,0){\includegraphics[width=\unitlength,page=2]{dual-dline-diagram.pdf}}%
    \put(0.71125689,0.33371343){\color[rgb]{1,0.4,0}\makebox(0,0)[lt]{\lineheight{1.25}\smash{\begin{tabular}[t]{l}$D'^4_{4,0}\brkt{0}{14}$\end{tabular}}}}%
    \put(0.48127258,0.40251253){\color[rgb]{0,0,0}\makebox(0,0)[lt]{\lineheight{1.25}\smash{\begin{tabular}[t]{l}$R^4$\end{tabular}}}}%
    \put(0,0){\includegraphics[width=\unitlength,page=3]{dual-dline-diagram.pdf}}%
  \end{picture}%
\endgroup%
}
\caption{Diagram depicting a dual digital line \cref{def:dline}.}
\label{fig:ddline}
\end{figure}

Let $\delta^m_{i,j}$ denote the Kronecker delta. For $m \in \NN$, $(i,j),(i',j') \in \ZZ \times I_{2^n}$, let $ \delta^m_{i,j} \in \cF^m$ be given by $\delta^m_{i,j}(i',j') =1$ if $(i,j) = (i',j')$ and $\delta^m_{i,j}(i',j') =0$ otherwise.

The digital line $D^n_m$ and its dual $D'^n_m$ can be expressed in terms of $R'^n_m$ and $R^n_m$, respectively.
 
\begin{lemma}
The digital line $D$ and its dual $D'$ are related by 
\beq
    \indic_{D'^{n}_{m} \brkt{i}{j} } = R^n_{m}[\delta^n_{i,j}],
    \quad
    \indic_{D^{n}_{m} \brkt{h}{s} } =  R'^n_{m}[\delta^n_{h,s}].
\eeq
\end{lemma} 

\begin{proof}
Since $(i,j) \in \nml{D}{\ell} \brkt{h}{s}$ if and only if $(h,s) \in \nml{D'}{\ell}\brkt{i}{j}$,
\beq
\nml{R}{\ell} [ \nmlz{\delta}{i}{j} ] (h,s)
= \indic_{\nml{D}{\ell} \brkt{h}{s}}(i,j)
= \indic_{\nml{D'}{\ell} \brkt{i}{j}} (h,s)
= \nml{R'}{\ell}[\nmlz{\delta}{h}{s}](i,j).
\eeq
\end{proof}

Therefore, one can also express the digital line $D^n_{m,\ell}$ and its dual $D'^n_{m,\ell}$ in terms of the back-projection $R'^n_m$ and the single-quadrant ADRT $R^n_m$,
\beq
    D^{n}_{m,\ell} \brkt{h}{s} = \supp R'^n_{m}[\delta^n_{h,s}],
    \quad
    D'^{n}_{m,\ell} \brkt{i}{j} = \supp R^n_{m}[\delta^n_{i,j}].
\eeq
With these observations at hand, it is straightforward to show that $R'^n$ is the transpose of $R^n$.

\begin{lemma}
$R'^n$ is the transpose of $R^n$.
\end{lemma}

\begin{proof}
From the definitions,
\beq
    \begin{aligned}
    (h,s) \in D'^n \brkt{i}{j}
    &\,\, \Leftrightarrow \,\,
    \delta^n_{h,s} \cdot \indic_{D'^n \brkt{i}{j}} = 1\\
    &\,\, \Leftrightarrow \,\,
    (i,j) \in D^n \brkt{h}{s}
    \,\, \Leftrightarrow \,\,
    \delta^n_{i,j} \cdot \indic_{D^n \brkt{h}{s}} = 1.
    \end{aligned}
\eeq
Then directly,
\beq
    \delta^n_{h,s} \cdot R^n [\delta^n_{i,j}] 
    = 
    \delta^n_{h,s} \cdot \indic_{D'^n \brkt{i}{j} }
    = 
    \indic_{D^n \brkt{h}{s} } \cdot \delta^n_{i,j} 
    = 
    \delta^n_{i,j} \cdot R'^n [\delta^n_{h,s}].
\eeq
\end{proof}
In the same manner in which $R^n_m$ was expressed as a composition of a sequence of $S^n_m$, the transpose $R'^n_m$ also attains an expression as compositions. Define the linear operator $B_m: \cF^{m+1} \to \cF^m \times \cF^m$ which reverses the addition of $S_m$ \cref{eq:Sm}: we sum the terms of $g$ that contribute to $g_0(h,t)$ which are $g(h,2t+1)$ and $g(h,2t)$, and those that contribute to $g_1(h,t)$, which are $g(h-t-1,2t+1)$ and $g(h-t,2t)$. That is,
\beq
    B_m [g]
    :=
    \sqbrkt{g_0}{g_1},
    \quad
    \left\{
\begin{aligned}
    g_0(h,t)
    &:=
    g(h,2t) + g(h, 2t+1),
    \\
    g_1(h,t)
    &:=
    g(h-t,2t) + g(h-t-1, 2t+1).
\end{aligned}\right.
\eeq 
We define an identical operation $B^n_m: \cF^n \to \cF^n$ over larger images by
\beq
    B^n_m[f] := g,
    \quad
    \sqbrkt{g^n_{m-1,2\ell'}}{g^n_{m-1,2\ell'+1}}
    = B_m[f^n_{m,\ell'}],
    \quad
    \ell' \in I_{2^{n-m-1}}.
\eeq 
It follows that $R'^n_{m,\ell}$ is a repeated application of $B^n_m$,
\beq
    R'^n_m[g] = B^n_1 \circ \cdots \circ B^n_{m} [g],
    \Ffor
    m \in I_{n+1}.
\eeq
Naturally, the dual digital line $D'$ satisfies a recursive relation that is very similar to that of $D$.

\begin{corollary}
A \emph{dual digital line} $D'^n_{m,\ell} \brkt{h}{s}$ for $(h,s) \in \ZZ \times \oI_{2^m}$ is a subset of $\ZZ \times \oI_{2^n}$ that satisfies
\beq
    \left\{
    \begin{aligned}
    D'^n_{m,\ell} \brkt{h}{2t}
      &:= D'^n_{m-1, 2\ell} \brkt{h}{t}
      \cup D'^n_{m-1, 2\ell+1} \brkt{h-t}{t},
      \\
    D'^n_{m,\ell} \brkt{h}{2t+1}
      &:= D'^n_{m-1, 2\ell} \brkt{h}{t}
      \cup D'^n_{m-1, 2\ell+1} \brkt{h-t-1}{t},
    \end{aligned}
    \right.
    \label{eq:ddline}
\eeq
and the relation is initialized by $D_{0,\ell}^n \brkt{h}{\ell} := \{(h,\ell)\},$ for $h \in \ZZ$, $\ell \in I_{2^n}$.
\end{corollary} 

One may derive direct formulas for the digital lines and their duals, without the use of a recursion. Perhaps not so efficient to use for computation, they reveal that the two definitions they are identical up to the sign of its increments.

\begin{corollary} 
Let $(s)_2$ denote the binary representation of $s \in \NN$,
$\bitand$ the bit-wise \emph{and},
$\bitnot$ the bitiwise \emph{not},
$\bitiprod$ the bitwise inner product,
and $\overline{(s)}_2$ the bit reversal of $(s)_2$.
\begin{enumerate}[label=(\roman*)]
\item $D^n_m \brkt{h}{s} = \{(k_+(t),t) \in \ZZ \times \oI_{2^n} \}$ in
which $k_+: \oI_{N} \to \ZZ$ satisfies
\beq
    \left\{ \begin{aligned}
    k_+(t) - k_+(t-1) 
    &=  ((t)_2 \bitand \bitnot (t-1)_2) \bitiprod \overline{(s)}_2,
    \\
    k_+(0) &= h.
    \end{aligned} \right.
    \label{eq:dline-inc}
\eeq
\item $D'^n_m \brkt{h}{s} = \{(k_-(t),t) \in \ZZ \times I_N \}$ in which $k_-:
\oI_{N} \to \ZZ$ satisfies 
\beq
    \left\{ \begin{aligned}
    k_-(t-1) - k_-(t) 
    &=  ((t)_2 \bitand \bitnot (t-1)_2) \bitiprod \overline{(s)}_2,
    \\
    k_+(0) &= h.
    \end{aligned} \right.
\eeq
\end{enumerate}
\end{corollary}

The increments in \cref{eq:dline-inc} for a sequence of digital lines, for whom the ratio $s/(N-1)$ is constant, is shown in \cref{fig:dline-inc}. The distribution of the increments is reminiscent of the low-discrepancy sequences \cite{NRbook92}.

\begin{figure} \label{fig:dline-inc}
\centering
\includegraphics[width=0.9\textwidth]{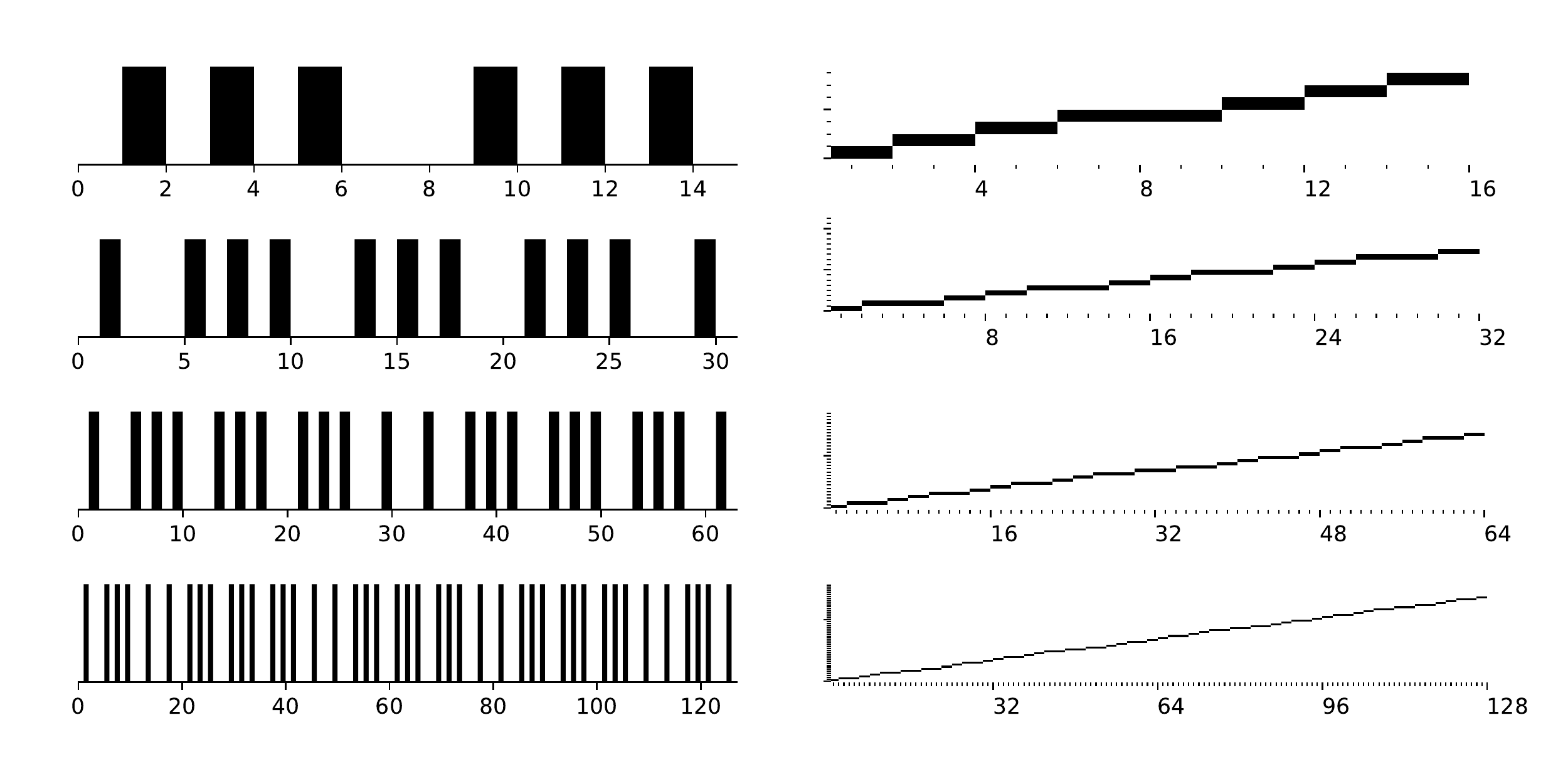}
\caption{Diagram depicting the increments (left) of corresponding
digital lines (right).}
\end{figure}

\subsection{Exact inverse} An important property of the single-quadrant ADRT is that it has an exact and fast inversion formula.

\begin{theorem}[Rim \cite{Rim20iadrt}]
If $f \in \cF^n_0$ then $R^n_{m-1,2\ell}[f]$ and $R^n_{m-1,2\ell+1}[f]$ can be computed from $R^n_{m,\ell}[f]$ for all $m \in I_{n+1}, \ell \in I_{2^{n-m}}$ in $\cO(N^2)$ operations.
\end{theorem}

We will briefly review the inversion formula. Let us define the differences, 
\beq
   \Delta^n_{m,\ell}[f] (h,s) := R^n_{m,\ell}[f] (h+1,s)
                            - R^n_{m,\ell}[f] (h,s),
\eeq
then these differences can be computed by the relations,
for all $\ell' \in I_{2^{n-m-1}}$
\begin{align}
    \Delta^n_{m-1,2\ell'}[f] (h,s) &= R^n_{m,\ell'}[f](h+1,2s)
                                  - R^n_{m,\ell'}[f](h,2s+1),
    \label{eq:diff_even}\\
    \Delta^n_{m-1,2\ell'+1}[f] (h,s) &= R^n_{m,\ell'}[f](h-s,2s+1)
                                  - R^n_{m,\ell'}[f](h-s,2s).
    \label{eq:diff_odd}
\end{align}
Using the differences, one recovers the $(m-1,\ell)$-single-quadrant ADRT,
\beq
    R^\ell_{m-1,\ell} [f] (h,s) 
    = \sum_{k=-s}^{h-1} \Delta^n_{m-1,\ell}[f] (k,s).
\eeq
So we obtain the formulae for $\ell' \in I_{2^{n-m-1}}$,
\beq
\begin{aligned}
    R^n_{m-1,2\ell'}[f] (h,s) 
    &= \sum_{k=-s}^{h-1} \left[ 
            R^n_{m,\ell'}[f](k+1,2s) - R^n_{m,\ell'}[f](k,2s+1)
            \right],
    \\
    R^n_{m-1,2\ell'+1}[f] (h,s) 
    &= \sum_{k=-s}^{h-1} \left[ 
            R^n_{m,\ell'}[f](k-s,2s+1) - R^n_{m,\ell'}[f](k-s,2s)
            \right],
    \label{eq:inv-formula}
\end{aligned}
\eeq
which yields the theorem.

This formula is readily extended to apply to images in $\cF^n_+$. Then, the operation given by this formula is the inverse of $S_m: \cF^n_+ \to \cF^n_+$. Let us use a notation for the operation \cref{eq:inv-formula}, say $S^{-1}_{m}:\cF^m_+ \to \cF^{m-1}_+ \times \cF^{m-1}_+$. That is, for $g \in \cF^m_+$,
\beq \label{eq:inv-formula+}
    \begin{aligned}
    S^{-1}_m[g]
    &:=
    \sqbrkt{g_0}{g_1},
    \\
    &
    \left\{
    \begin{aligned}
    g_0(h,s) &=
    \sum_{k=-\infty}^{h-1} [g(k+1, 2s) - g(k,2s+1)],
    \\
    g_1(h,s) &=
    \sum_{k=-\infty}^{h-1} [g(k-s, 2s+1) - g(k-s,2s)],
    \end{aligned}
    \right.
    \end{aligned}
\eeq
then the inversion formula \cref{eq:inv-formula} can be written as
\beq
    S^{-1}_{m} [R^n_{m,\ell'}[f]] 
    = 
    \sqbrkt{R^n_{m-1,2\ell'}[f]}{R^n_{m-1,2\ell'+1}[f]}.
    \label{eq:Smi}
\eeq
Let $(S_m^n)^{-1}: \cF^n_+ \to \cF^n_+$ be defined
\beq
    (S^n_m)^{-1}[g] := f
    \with
    \sqbrkt{f^n_{m-1,2\ell'}}{f^n_{m-1,2\ell'+1}}
    =
    S_m^{-1}[f^n_{m,\ell'}],
    \quad
    \ell \in I_{2^{n-m}}.
\eeq
Then we have
\beq
    (R^n_m)^{-1}[g]
    =
    (S_1^n)^{-1} \circ \cdots \circ (S_m^n)^{-1} [g].
\eeq

\begin{corollary}
The inverse of single-quadrant ADRT $(R^n)^{-1}: R^n[\cF^n_0] \to \cF_0^n$ is given by
\beq
    (R^n)^{-1}[g]
    =
    (S_1^n)^{-1} \circ \cdots \circ (S_n^n)^{-1} [g].
    \label{eq:Rni}
\eeq
\end{corollary}

Thus, the computational complexity of computing $R^n$ is equal to that of computing $(R^n)^{-1}$, both $\cO(N^2 \log N)$. 

The direction of the sum in \cref{eq:inv-formula+} is not a unique choice. If so inclined, one may sum from $+\infty$ to $h$, granted the original image $f$ is taken from the half-vertical strip that is unbounded in the negative $i$-direction, say $\cF_-^n$. In the case $f \in \cF_0$, one may also take an average of two sums ranging from $-\infty$ to $\infty$. We will continue with one-sided sum here, and in \cref{sec:range}, it will become clear that this choice coincides with a back-substitution formula.

\section{Range characterization} \label{sec:range}
This section contains the range characterization of the single-quadrant ADRT over square images $\cF^n_0$. We proceed by first showing that $R^n: \cF^n_+ \to \cF^n_+$ is a bijection, due to the exact inversion formula. Next, we find $\cF^n_R \subset \cF^n_+$ for which $R^n: \cF^n_0 \to \cF^n_R$ is a bijection, thereby characterizing the range of $R^n$ over square images.

\subsection{Characterization of $R^n[\cF^n_+]$} The inversion formula \cref{eq:inv-formula} leads to the fact that $R^n$ is bijective. This is summed up in a corollary.

\begin{corollary}
$R^n (R^n)^{-1} [\delta^n_{h,s}] = \delta^n_{h,s}$ and
$(R^n)^{-1} R^n [\delta^n_{i,j}] = \delta^n_{i,j}$.
\end{corollary} 

\begin{proof}
By the virtue of \cref{eq:Smi},
$(S_k) \circ (S_k)^{-1} [\delta^{n-k+1}_{h,s}] = \delta^{n-k+1}_{h,s}$,
so that
\[
    R^n (R^n)^{-1} [\delta^n_{h,s}]
    = (S^n_n) \circ \cdots \circ (S^n_1)
    \circ (S^n_1)^{-1} \circ \cdots \circ (S^n_n)^{-1} [\delta^n_{h,s}]
    =  \delta^n_{h,s}.
\]
Similarly, the latter identity follows from
\[
    \begin{aligned}
    (S_k)^{-1} \circ (S_k)
    \sqbrkt{\delta^{n-k+1}_{h,s}}{0}
    &=
    \sqbrkt{\delta^{n-k+1}_{h,s}}{0},
    \\
    \quad
        (S_k)^{-1} \circ (S_k)
    \sqbrkt{0}{\delta^{n-k+1}_{h,s}}
    &=
    \sqbrkt{0}{\delta^{n-k+1}_{h,s}},
    \end{aligned}
\]
which can be verified from direct computation, for example,
\[
    \begin{aligned}
    (S_k)^{-1} \circ (S_k) &
    \sqbrkt{\delta^{n-k+1}_{h,s} }{0}
    \\&=
    (S_k)^{-1} 
    \left[ \delta^{n-k+2}_{h,2s} + \delta^{n-k+2}_{h,2s+1} \right]
    \\&= 
    \sqbrkt{\sum_{k=-\infty}^{h' - 1} \delta^{n-k+1}_{h,s} (k+1,s')
    -
    \sum_{k=-\infty}^{h' - 1} \delta^{n-k+1}_{h,s} (k,s')}{0}
    \\
    &= \sqbrkt{\delta^{n-k+1}_{h,s}(h',s')}{0},
    \end{aligned}
\]
and the second equality follows similarly.
\end{proof}

Furthermore, we will show that $R^n[\cF^n_+] = \cF^n_+$ by providing an explicit
expression for $(R^n)^{-1}[\delta^n_{h,s}]$ for each $(h,s) \in \ZZ \times
I_{2^n}$. First, it will be expedient to define the mapping $\lambda: \ZZ \times
\NN \times \{\pm 1\} \to (\ZZ \times \NN \times \{\pm 1\})^2$ 
\beq
    \lambda(h,s,\sigma) 
    := 
    \sqbrkt{\lambda_0}{\lambda_1},
    \quad
    \left\{
    \begin{matrix*}[l]
     \lambda_0 := & \big[h + (s \brem 2) , 
                  &  \lfloor s/2 \rfloor, 
                  & \sigma \cdot (-1)^{(s \brem 2)} \big],
    \\
     \lambda_1 := & \big[ h + \lfloor s/2 \rfloor + 1, 
                  & \lfloor s/2 \rfloor, 
                  & \sigma \cdot (-1)^{(s \brem 2)+1} \big].
     \end{matrix*}
    \right.
    \label{eq:lambda-def}
\eeq
It is straightforward to show that $\lambda$ is an injective map, as the pair
$\sqbrkt{\lambda_0}{\lambda_1}$ uniquely defines the triple $(h,s,\sigma)$ for
which $\lambda(h,s,\sigma) = \sqbrkt{\lambda_0}{\lambda_1}$: denoting the three
entries by
    $\lambda_0 = [\lambda_{0,0}, \lambda_{0, 1}, \lambda_{0, 2}],$
    $\lambda_1 = [\lambda_{1,0}, \lambda_{1, 1}, \lambda_{1, 2}]$
one recovers e.g. $h = \lambda_{1, 0} - \lambda_{1,1} - 1$, $s = \lambda_{0, 0}
+ 2\lambda_{1, 0} - 3h - 2$, $\sigma = \lambda_{0,2} \cdot (-1)^{(s \brem 2)}$
successively.

Below, let $(\cdot)^i_+: \RR \to \RR_+$ is given by $(x)^i_+ := \max\{0,x^i\}$ for $i \in \NN$. When $i=1$ it is also called the rectified linear unit (ReLU), and when $i=0$ it is the Heaviside step function.

\newpage

\begin{lemma} \label{lem:delta-1lvl}
Let $(h,s,\sigma) \in \ZZ \times \oI_{2^m} \times \{\pm1\}$ and let 
\beq
\lambda_0 = \sqbrkte{h_0}{s_0}{\sigma_0},
\quad
\lambda_1 = \sqbrkte{h_1}{s_1}{\sigma_1},
\quad
\text{where }
\sqbrkt{\lambda_0}{\lambda_1} = \lambda[h,s,\sigma].
\eeq
Then $g_0,g_1 \in \cF^m_+$ of the inverse $(S_m)^{-1} [\sigma \delta^m_{h,s}] = \sqbrkt{g_0}{g_1}$ are given by the formula
\beq
    g_b(h,s)
    =
    \sigma_b \sum_{h' \ge h_b} \delta^{m-1}_{h',s_b} (h,s)
    =
    \sigma_b \delta^{m-1}_{s_b}(s) (h - h_b)_+^0,
    \quad
    b \in \{0,1\}.
    \label{eq:iSm-delta}
\eeq
\end{lemma} 

\begin{proof}
In the case $s$ is even, let $s = 2t$ then
\beq
    \left\{
    \begin{aligned}
    g_0(h',s')
    &=
    \sum^{h'-1}_{k = -s'} \delta^{m-1}_{h,s} (k+1, 2s')
    =
    \delta_{t}(s') (h' - h)_+^0,
    \\
    g_1(h',s')
    &=
    -\sum^{h'-1}_{k = -s'} \delta^{m-1}_{h,s} (k+1, 2s')
    =
    -\delta_{t}(s') (h' - h - t - 1)_+^0    
    \end{aligned}
    \right.
\eeq
In case $s$ is odd, let $s = 2t + 1$,
\beq
    \left\{
    \begin{aligned}
    g_0(h',s')
    &=
    -\sum_{k= -s'}^{h'-1} \delta^{m-1}_{h,s} (k, 2s' + 1) 
    = 
    -\delta_{t}(s') (h' - h- 1)_+^0,
    \\
    g_1(h',s')
    &=
    -\sum^{h'-1}_{k = -s'} \delta^{m-1}_{h,s} (k-s', 2s')
    =
    \delta_{t}(s') (h' - h - t -1)_+^0.
    \end{aligned}
    \right.
\eeq
These direct results are summarized via \cref{eq:lambda-def} and \cref{eq:iSm-delta}.
\end{proof}
 Observe that $\supp S_m^{-1} [\delta^m_{h,s}]$ is unbounded in the positive $h$-direction and so belongs to $\cF^n_+$.  The vertical half-lines on which the inverse is supported is denoted by,
\beq
        L^m(h,s)
        :=
        \left\{ 
        (h',s') \in \ZZ \times I_{2^m} :
            h' \ge h, s' = s 
            \right\},
\eeq
so we can also write $(S_m)^{-1} [\sigma \cdot \delta^m_{h,s}] = \sqbrkt{\sigma_0 \cdot \indic_{L^{m-1}(h_0,s_0)}}{ \sigma_1 \cdot \indic_{L^{m-1}(h_1,s_1)}}$.

A plot of the inverse $(S^n_n)^{-1}[\delta^n_{h,s}]$ is shown in \cref{fig:Snm-delta}: the inverted image consists of two slices of vertically oriented jumps.

\begin{figure}
\centering
\def\svgwidth{0.9\textwidth}
{\scriptsize 
\begingroup%
  \makeatletter%
  \providecommand\color[2][]{%
    \errmessage{(Inkscape) Color is used for the text in Inkscape, but the package 'color.sty' is not loaded}%
    \renewcommand\color[2][]{}%
  }%
  \providecommand\transparent[1]{%
    \errmessage{(Inkscape) Transparency is used (non-zero) for the text in Inkscape, but the package 'transparent.sty' is not loaded}%
    \renewcommand\transparent[1]{}%
  }%
  \providecommand\rotatebox[2]{#2}%
  \newcommand*\fsize{\dimexpr\f@size pt\relax}%
  \newcommand*\lineheight[1]{\fontsize{\fsize}{#1\fsize}\selectfont}%
  \ifx\svgwidth\undefined%
    \setlength{\unitlength}{679.01012046bp}%
    \ifx\svgscale\undefined%
      \relax%
    \else%
      \setlength{\unitlength}{\unitlength * \real{\svgscale}}%
    \fi%
  \else%
    \setlength{\unitlength}{\svgwidth}%
  \fi%
  \global\let\svgwidth\undefined%
  \global\let\svgscale\undefined%
  \makeatother%
  \begin{picture}(1,0.82538112)%
    \lineheight{1}%
    \setlength\tabcolsep{0pt}%
    \put(0,0){\includegraphics[width=\unitlength,page=1]{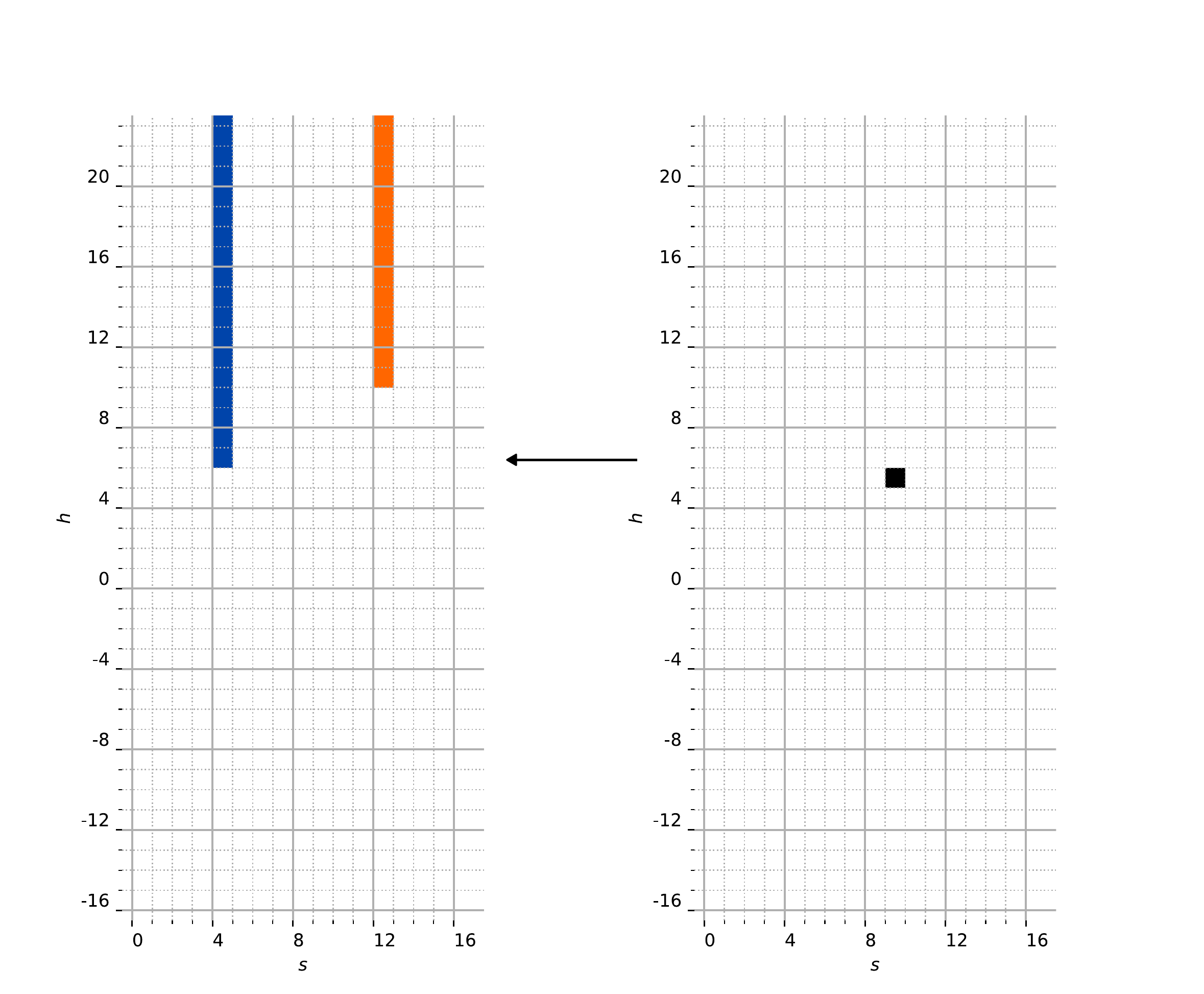}}%
    \put(0.45012395,0.45866189){\color[rgb]{0,0,0}\makebox(0,0)[lt]{\lineheight{1.25}\smash{\begin{tabular}[t]{l}$(S^n_n)^{-1}$\end{tabular}}}}%
    \put(0.33083264,0.50947115){\color[rgb]{1,0.4,0}\makebox(0,0)[lt]{\lineheight{1.25}\smash{\begin{tabular}[t]{l}$+$\end{tabular}}}}%
    \put(0.20049585,0.4409891){\color[rgb]{0,0.26666667,0.66666667}\makebox(0,0)[lt]{\lineheight{1.25}\smash{\begin{tabular}[t]{l}$-$\end{tabular}}}}%
    \put(0.22037773,0.75247196){\color[rgb]{0,0,0}\makebox(0,0)[lt]{\lineheight{1.25}\smash{\begin{tabular}[t]{l}$S_n [\delta^n_{h,s}]$\end{tabular}}}}%
    \put(0.70417025,0.75247196){\color[rgb]{0,0,0}\makebox(0,0)[lt]{\lineheight{1.25}\smash{\begin{tabular}[t]{l}$\delta^n_{h,s}$\end{tabular}}}}%
  \end{picture}%
\endgroup%
}
\caption{A plot of $(S_n^n)^{-1} [\delta^n_{h,s}]$ given by the formula \cref{eq:Smi}.}
\label{fig:Snm-delta}
\end{figure}

Then \cref{lem:delta-1lvl} and the decomposition of $R^n$ \cref{eq:Rni} enable us to compute $(R^n)^{-1} [\delta^n_{h,s}]$. A repeated application of $\lambda$ \cref{eq:lambda-def} yields,
\beq \label{eq:lambda-t}
    \begin{aligned}
    \lambda_0 = (h_0,s_0,\sigma_0),
    \quad &\cdots, \quad 
    \lambda_{2^n-1} = (h_{2^n-1},s_{2^n-1},\sigma_{2^n-1}),
    \\
    (\lambda_0, ... , \lambda_{2^n-1})
    &:=
    \underbrace{ \lambda \otimes ... \otimes \lambda}_{n \text{ times}} (h,s).
    \end{aligned}
\eeq
For $t \in I_{2^m}$, we may rewrite $\lambda_t$ in terms of the binary expansion $t = (b_0 b_1 \cdots b_k)_2$ by the shorthand $\lambda_t = \lambda_{b_0 b_1 \cdots b_k}$. Using this, one expresses diagrammatically the repeated application of the map $\lambda$ \cref{eq:lambda-def},
\vspace{-0.6cm} %
\begin{center}
\def\svgwidth{0.4\textwidth}
{ 
\begingroup%
  \makeatletter%
  \providecommand\color[2][]{%
    \errmessage{(Inkscape) Color is used for the text in Inkscape, but the package 'color.sty' is not loaded}%
    \renewcommand\color[2][]{}%
  }%
  \providecommand\transparent[1]{%
    \errmessage{(Inkscape) Transparency is used (non-zero) for the text in Inkscape, but the package 'transparent.sty' is not loaded}%
    \renewcommand\transparent[1]{}%
  }%
  \providecommand\rotatebox[2]{#2}%
  \newcommand*\fsize{\dimexpr\f@size pt\relax}%
  \newcommand*\lineheight[1]{\fontsize{\fsize}{#1\fsize}\selectfont}%
  \ifx\svgwidth\undefined%
    \setlength{\unitlength}{671.51695767bp}%
    \ifx\svgscale\undefined%
      \relax%
    \else%
      \setlength{\unitlength}{\unitlength * \real{\svgscale}}%
    \fi%
  \else%
    \setlength{\unitlength}{\svgwidth}%
  \fi%
  \global\let\svgwidth\undefined%
  \global\let\svgscale\undefined%
  \makeatother%
  \begin{picture}(1,0.71525087)%
    \lineheight{1}%
    \setlength\tabcolsep{0pt}%
    \put(0.63821383,0.59718131){\color[rgb]{0,0,0}\makebox(0,0)[lt]{\begin{minipage}{0.09732762\unitlength}\raggedright \end{minipage}}}%
    \put(0.09893027,0.26406152){\color[rgb]{0,0,0}\makebox(0,0)[lt]{\lineheight{1.25}\smash{\begin{tabular}[t]{l}$\lambda_0$\end{tabular}}}}%
    \put(0,0.71525087){\color[rgb]{0,0,0}\makebox(0,0)[lt]{\lineheight{1.25}\smash{\begin{tabular}[t]{l} \end{tabular}}}}%
    \put(0,0){\includegraphics[width=\unitlength,page=1]{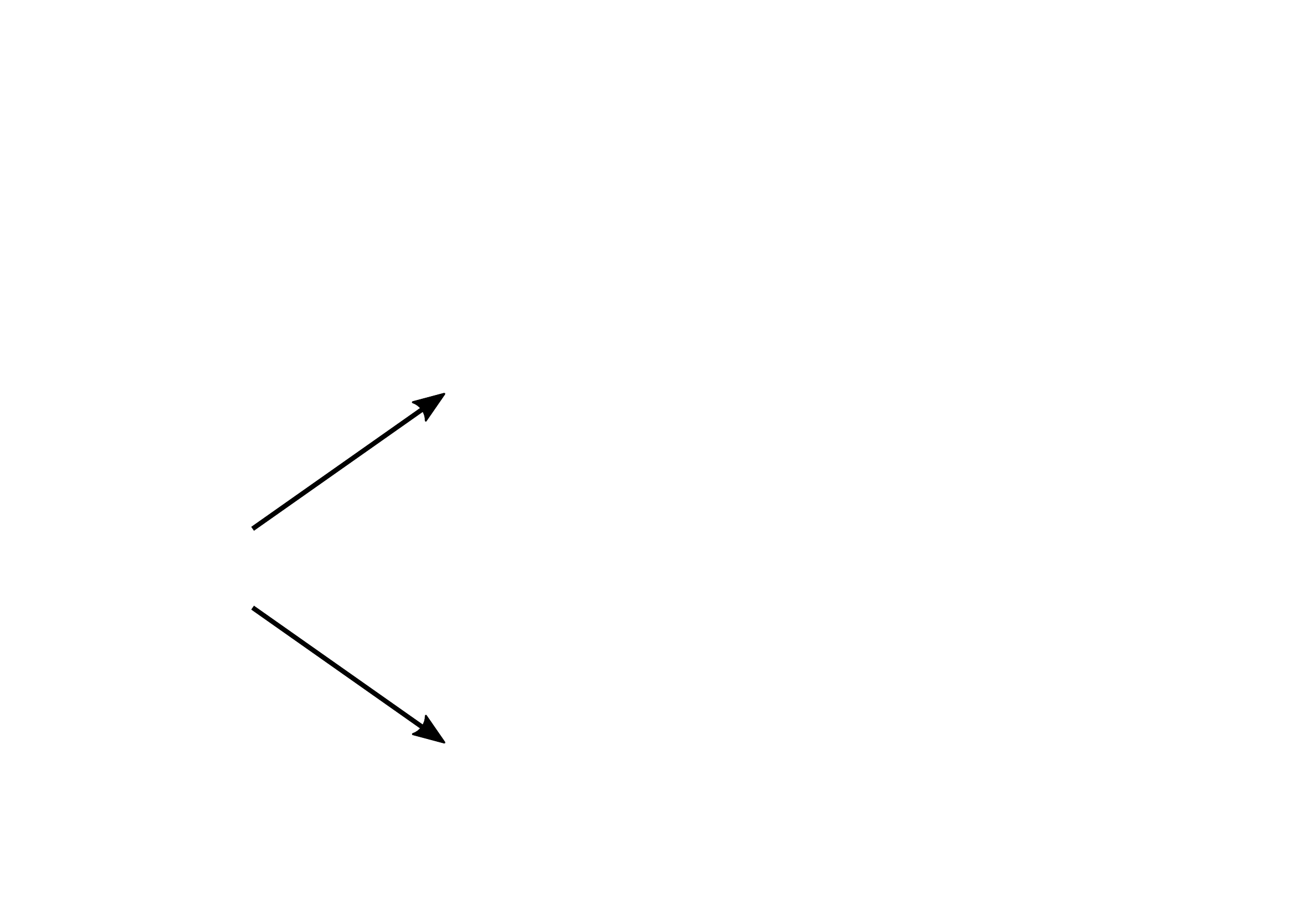}}%
    \put(0.3624848,0.41877906){\color[rgb]{0,0,0}\makebox(0,0)[lt]{\lineheight{1.25}\smash{\begin{tabular}[t]{l}$\lambda_{10}$\end{tabular}}}}%
    \put(0.3624848,0.11563626){\color[rgb]{0,0,0}\makebox(0,0)[lt]{\lineheight{1.25}\smash{\begin{tabular}[t]{l}$\lambda_{11}$\end{tabular}}}}%
    \put(0,0){\includegraphics[width=\unitlength,page=2]{lambda-diagram.pdf}}%
    \put(0.61240395,0.51242413){\color[rgb]{0,0,0}\makebox(0,0)[lt]{\lineheight{1.25}\smash{\begin{tabular}[t]{l}$\lambda_{100}$\end{tabular}}}}%
    \put(0.61240395,0.31514053){\color[rgb]{0,0,0}\makebox(0,0)[lt]{\lineheight{1.25}\smash{\begin{tabular}[t]{l}$\lambda_{101}$\end{tabular}}}}%
    \put(0.61240395,0.21187206){\color[rgb]{0,0,0}\makebox(0,0)[lt]{\lineheight{1.25}\smash{\begin{tabular}[t]{l}$\lambda_{110}$\end{tabular}}}}%
    \put(0.61240395,0.00977684){\color[rgb]{0,0,0}\makebox(0,0)[lt]{\lineheight{1.25}\smash{\begin{tabular}[t]{l}$\lambda_{111}$\end{tabular}}}}%
    \put(0.83257715,0.00793485){\color[rgb]{0,0,0}\makebox(0,0)[lt]{\lineheight{1.25}\smash{\begin{tabular}[t]{l}$\cdots$\end{tabular}}}}%
    \put(0.83257715,0.21003014){\color[rgb]{0,0,0}\makebox(0,0)[lt]{\lineheight{1.25}\smash{\begin{tabular}[t]{l}$\cdots$\end{tabular}}}}%
    \put(0.83257715,0.31329847){\color[rgb]{0,0,0}\makebox(0,0)[lt]{\lineheight{1.25}\smash{\begin{tabular}[t]{l}$\cdots$\end{tabular}}}}%
    \put(0.83257715,0.51242413){\color[rgb]{0,0,0}\makebox(0,0)[lt]{\lineheight{1.25}\smash{\begin{tabular}[t]{l}$\cdots$\end{tabular}}}}%
    \put(0.26675278,0.2605109){\color[rgb]{0,0,0}\makebox(0,0)[lt]{\lineheight{1.25}\smash{\begin{tabular}[t]{l}$\lambda$\end{tabular}}}}%
  \end{picture}%
\endgroup%
}
\end{center}
The indices $\{h_t\}_{t=0}^{2^n-1}$ then provide the lower bounds of the support of $(R^n)^{-1}[\delta^n_{h,s}]$. A little more precisely, the support equals the union of vertical half-lines
\beq
  L^n_m(h,s)
  :=
  \bigcup_{t \in I_{2^{n-m}}}
  L^n(h_t,s_t).
\eeq

\begin{theorem} \label{thm:Fp-Fp}
$R^n$ is a bijection from $\cF^n_+$ onto itself. In particular, for each triple $(h,s,\sigma) \in \ZZ \times I_{2^n} \times \{\pm 1\}$,
\beq \label{eq:iRndelta}
    (R^n)^{-1} [\sigma \delta^n_{h,s}] (i,j)
    = 
    \sigma_j (i - h_j)_+^{n-1},
    \quad
    (i,j) \in \ZZ \times I_{2^n},
\eeq
in which $(h_j, \sigma_j) \in \ZZ \times \{\pm 1\}$ as given in \cref{eq:lambda-t}.
\end{theorem} 

\begin{figure}
\centering
\def\svgwidth{0.8\textwidth}
{\scriptsize 
\begingroup%
  \makeatletter%
  \providecommand\color[2][]{%
    \errmessage{(Inkscape) Color is used for the text in Inkscape, but the package 'color.sty' is not loaded}%
    \renewcommand\color[2][]{}%
  }%
  \providecommand\transparent[1]{%
    \errmessage{(Inkscape) Transparency is used (non-zero) for the text in Inkscape, but the package 'transparent.sty' is not loaded}%
    \renewcommand\transparent[1]{}%
  }%
  \providecommand\rotatebox[2]{#2}%
  \newcommand*\fsize{\dimexpr\f@size pt\relax}%
  \newcommand*\lineheight[1]{\fontsize{\fsize}{#1\fsize}\selectfont}%
  \ifx\svgwidth\undefined%
    \setlength{\unitlength}{2271.481794bp}%
    \ifx\svgscale\undefined%
      \relax%
    \else%
      \setlength{\unitlength}{\unitlength * \real{\svgscale}}%
    \fi%
  \else%
    \setlength{\unitlength}{\svgwidth}%
  \fi%
  \global\let\svgwidth\undefined%
  \global\let\svgscale\undefined%
  \makeatother%
  \begin{picture}(1,0.72088255)%
    \lineheight{1}%
    \setlength\tabcolsep{0pt}%
    \put(0,0){\includegraphics[width=\unitlength,page=1]{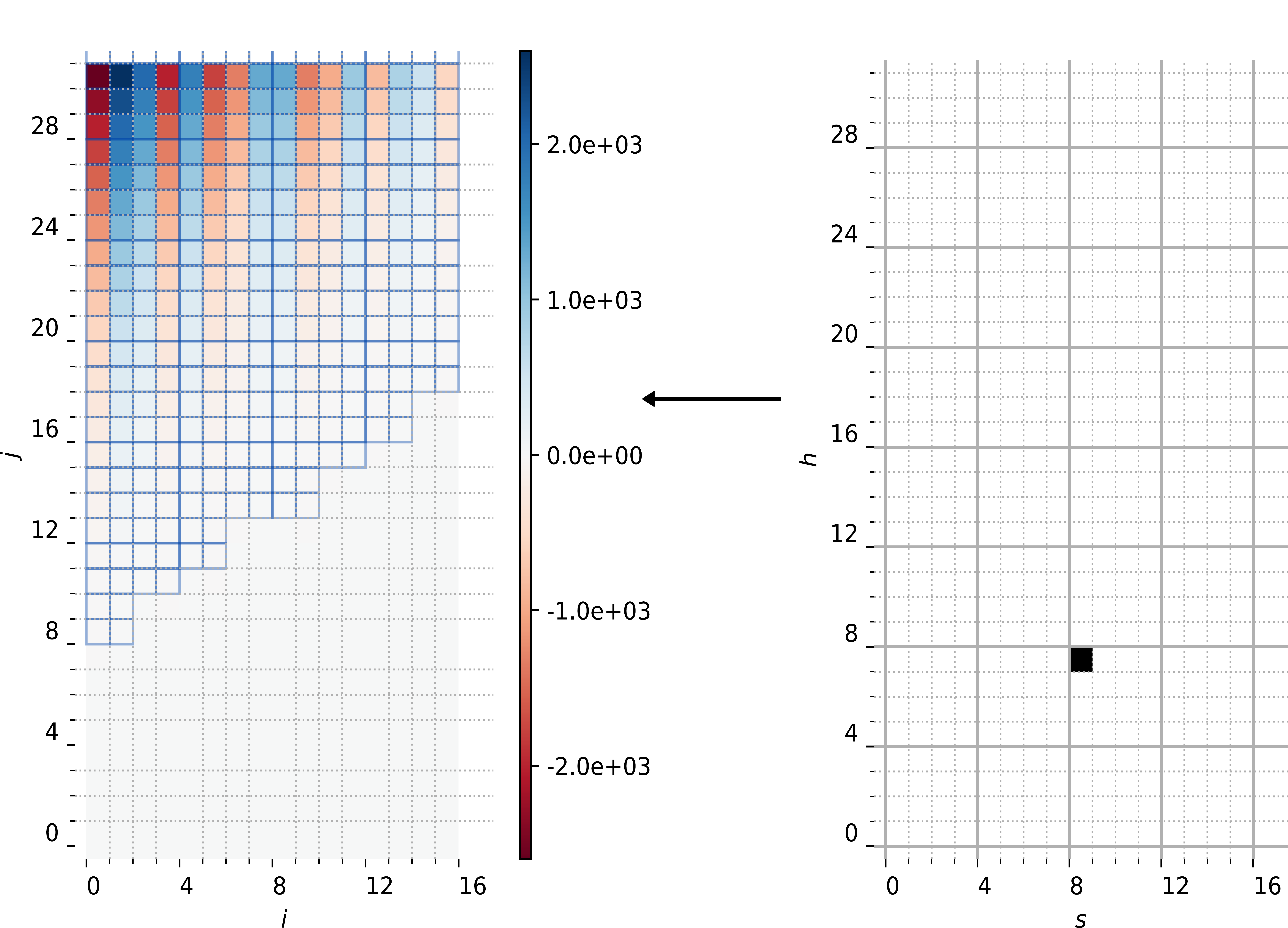}}%
    \put(0.52411958,0.43036438){\color[rgb]{0,0,0}\makebox(0,0)[lt]{\lineheight{1.25}\smash{\begin{tabular}[t]{l}$(R^n)^{-1}$\end{tabular}}}}%
    \put(0.15907806,0.70510182){\color[rgb]{0,0,0}\makebox(0,0)[lt]{\lineheight{1.25}\smash{\begin{tabular}[t]{l}$(R^n)^{-1} [\delta^n_{h,s}]$\end{tabular}}}}%
    \put(0.78747096,0.70510182){\color[rgb]{0,0,0}\makebox(0,0)[lt]{\lineheight{1.25}\smash{\begin{tabular}[t]{l}$\delta^n_{h,s}$\end{tabular}}}}%
    \put(0,0){\includegraphics[width=\unitlength,page=2]{Ridelta-diagram.pdf}}%
    \put(0.22426404,0.09614771){\color[rgb]{0,0,0}\makebox(0,0)[lt]{\lineheight{1.25}\smash{\begin{tabular}[t]{l}nonzero entry\end{tabular}}}}%
  \end{picture}%
\endgroup%
}
\caption{A plot of $(R^n)^{-1}[\delta^n_{h,s}]$ given by the formula \cref{eq:iRndelta}.}
\end{figure}
\begin{proof}
This follows from a repeated application of \cref{lem:delta-1lvl}.
Let us denote
\beq
    \begin{aligned}
    (S_n)^{-1} [\delta^n_{h,s}] &= \sqbrkt{g_0}{g_1}, \\
    (S_{n-1})^{-1} [g_0] &= \sqbrkt{g_{00}}{g_{01}},
    \quad
    (S_{n-1})^{-1} [g_1] = \sqbrkt{g_{10}}{g_{11}}, 
    \\
    &\vdots
    \end{aligned}
\eeq
so that we can denote $g_s$ using the binary expansion $s = (b_0 \cdots b_k)_2$ as follows,
\beq
(S_{n-k})^{-1} [g_{b_0 \cdots b_{k}}] 
= 
(g_{b_0 \cdots b_{k} b_{k+1}}, g_{b_0 \cdots b_{k} b_{k+1}})
\quad
b_k \in \{0,1\}.
\eeq
We proceed by induction. Suppose that it holds
\beq
\begin{aligned}
g_{b_0 \cdots b_{k}} (h',s')
&=
\sigma_{b_0 \cdots b_{k}} 
\delta^{n-k-1}_{s_{b_0 \cdots b_{k}}} (s')
(h' - h_{b_0 \cdots b_{k}} )_+^{k-1}
\\&=
\sigma_{b_0 \cdots b_k} 
\sum_{h \ge h_{b_0 \cdots b_k}} 
\delta^{n-k-1}_{h,s_{b_0 \cdots b_{k}}} (h',s')
(h' - h_{b_0 \cdots b_k} )_+^{k-1}.
\end{aligned}
\eeq
Then,
\beq
\begin{aligned}
&(S_{n-k})^{-1} [g_{b_0 \cdots b_k} ](h',s') \\
&=
(S_{n-k})^{-1} 
\left[
    \sigma_{b_0 \cdots b_k} 
    \sum_{h \ge s_{b_0 \cdots b_k}}
    \delta^{n-1}_{h,s_{b_0 \cdots b_k}} (h',s')
        (h' - h_{b_0 \cdots b_k})_+^{k-1}
\right]
\\&=
    \sum_{j=1}^\infty
    (j)^{k-1}
    (S_{n-k+1})^{-1} 
    \left[
        \sigma_{b_0 \cdots b_k} 
        \delta^{n-1}_{h_{b_0 \cdots b_k} + j -1,s_{b_0 \cdots b_k}} (h',s')
    \right]
\end{aligned}
\eeq
So it follows,
\beq
\begin{aligned}
 g_{b_0 \cdots b_{k+1}}(h',s') &=
    \sum_{j=1}^\infty
    (j)^{k-1}
    \sigma_{b_0 \cdots b_{k+1}}     
    \sum_{j'=1}^\infty
        \delta^{n-1}_{h_{b_0 \cdots b_{k+1}} + j -1 + j' -1,s_{b_0 \cdots b_{k+1}}} (h',s')
\\&=
    \sigma_{b_0 \cdots b_{k+1}} 
    \sum_{j=1}^\infty
    (j^{k-1} \cdot j)
    \delta^{n-1}_{h_{b_0 \cdots b_{k+1}} + j -1,s_{b_0 \cdots b_{k+1}}} (h',s')
\\&=
    \sigma_{b_0 \cdots b_{k+1}}     
    \sum_{h \ge h_{b_0 \cdots b_{k+1}}}
    \delta^{n-1}_{h,s_{b_0 \cdots b_{k+1}}} (h',s')
    (h - h_{b_0 \cdots b_{k+1}})^k_+.
\end{aligned}
\eeq
Now, the case $k =0$ follows by \cref{eq:iSm-delta}, proving the formula.

Since the map $\lambda$ is injective, the formula  \cref{eq:iRndelta} shows that $(R^n)^{-1}$ is injective. Thus $R^n$ is a bijection.
\end{proof}

It is worth noting that the continuous Radon transform does not possess an
analogous property. The continuous inverse is ill-defined when applied to
arbitrary functions in the unconstrained codomain \cite{Natterer}. In contrast,
one can extend the inverse ADRT $(R^n)^{-1}$ from $R^n[\cF^n_0]$ to the codomain
$\cF_+^n$, simply by enlarging the domain to allow infinitely supported images,
that is, by extending $R^n$ from $\cF^n_0$ to $\cF_+^n$. Accordingly, this
property of ADRT is unlikely to be preserved in the continuous limit.

The inverse image $(R^n)^{-1}[\delta^n_{h,s}]$ is plotted in \cref{fig:Snm-delta}. One observes an oscillatory behavior in the $i$-direction, and a polynomial growth in the $j$-direction. The oscillation can be related to the oscillatory behavior of the non-unique solutions to the limited angle problem for the continuous Radon transform \cite[Ch3]{Natterer}; the details of this relation will be pursued elsewhere.

Finally, we give a more explicit expression for the function $\lambda$.

\begin{corollary}
Letting $s^{(k)} := \lfloor s / 2^k \rfloor$, $s_r^{(k)} := (s^{(k-1)} \brem
2)$, if $t \in I_{2^n}$ has the binary representation $t = (b_1 b_2 ...
b_n)_2$, then $h_t,\sigma_t$ above in \cref{eq:lambda-t} can be written as
\beq
    h_t = h + \sum_{k=1}^n (b_k s^{(k)} + (1 - b_k) s^{(k)}_r),
    \quad
    \sigma_t = (-1)^{\bitsum (t \bitxor s)},
\eeq
where $\bitsum$ is the bit-wise \emph{sum}, and $\bitxor$ is the bit-wise
\emph{xor}.
\end{corollary}

\subsection{Conditional convergence for finitely supported images}

The fact that $R^n$ is a bijection from $\cF^n_+$ onto itself is dependent on the unboundedness of $\supp f$ for $f \in \cF^n_+$. The following observation shows that for $R^n$ to be bijective the sums \cref{eq:inv-formula} appearing in the inversion formula must converge conditionally, as the summands exhibit polynomial growth. Moreover, the degree of the polynomial growth depends on $n$, and hence the inverse diverges as $N \to \infty$.

Let $\indic_k := \indic_{I_{(k+1) 2^n} \times I_{2^n}}$ with $k \in \NN$ denote the cut-off function supported in a rectangle of width $2^n$. Then, 
\[
    \Norm{ R^n \indic_k (R^n)^{-1}  - \Id }{\infty} \gtrsim k^n
    \quad 
    k \ge 1,
\]
where $\Norm{\cdot}{\infty}$ denotes for $B : \cF^n_+ \to \cF^n_+$ the induced norm  $\Norm{B}{\infty} := \sup_{f \in \cF^n_+} \frac{\Norm{B f}{\infty}}{\Norm{f}{\infty}}$ with $\Norm{f}{\infty} := \sup_{(h,s) \in \ZZ \times I_{2^n}} |f(h,s)|$.

This can be seen as follows. Choose $\delta^n_{h,s}$ with $(h,s) = (-2^n+1, 2^n-1)$, and let $[h_{2^n -1}, s_{2^n-1}, \sigma_{2^n-1}]$ be defined as in \cref{eq:lambda-t}. We then have
\beq
    \begin{aligned}
    h_{2^n-1} 
    = h_0 
        + 
      \sum_{k=1}^n 
        \left(\left\lfloor \frac{s_k}{2} \right\rfloor +1 \right)
    = 
    -2^n + 1 + \sum_{k=1}^n [(2^{k-1} - 1) + 1] = 0,
    \end{aligned}
\eeq
where $s_{2^n-1} = 0$, $\sigma_{2^n-1} = 1$.

Consider the digital line $D^n\brkt{(k+1)2^n}{2^n-1}$. By \cref{eq:iRndelta}, $h' = (k+1)2^n$, $s' = 2^n-1$,
\[
    \left\lvert (R^n)^{-1} [\delta^n_{h,s}] (h', s') \right\rvert
    =
    ( (k+1) 2^n )^n 
    \ge 
    2^{2n} (k+1)^n
    \gtrsim k^n.
\]
In contrast, $\Norm{\indic_{k-1} \cdot (R^n \indic_k \cdot (R^n)^{-1} - \Id )}{\infty} = 0$ for $k \ge 1$, from the fact that all the digital lines that correspond to the points in  $\supp (\indic_{k-1} \cdot R^n)$ lie within $\supp \indic_k$.

\subsection{Localization Lemma} To characterize the range of the images in $f \in \cF^n_0$, we first prove some lemmas. Let us define the set, for each $m = 1, ... , n$,
\beq 
    E^n_{m,\ell}
    :=
    \Big\{
      (h,s) \subset \ZZ \times I_{2^n} \,|\,
      - (s \brem 2^m) \le h \le 2^n-1, \,
      s - \ell 2^m \in I_{2^m}
    \Big\},
    \label{eq:Enml}
\eeq
and define $E^n_{0,\ell} := I_{2^n} \times \{\ell\}$. We will denote $E^n := E^n_{n,0}$. and let $E^n_m := \bigcup_{\ell \in I_{2^{n-m}}} E^n_{m,\ell}$. A depiction of these sets appear in \cref{fig:Enml}.

\begin{figure}
\centering
\begin{tabular}{ccc}
\includegraphics[width=0.3\textwidth]{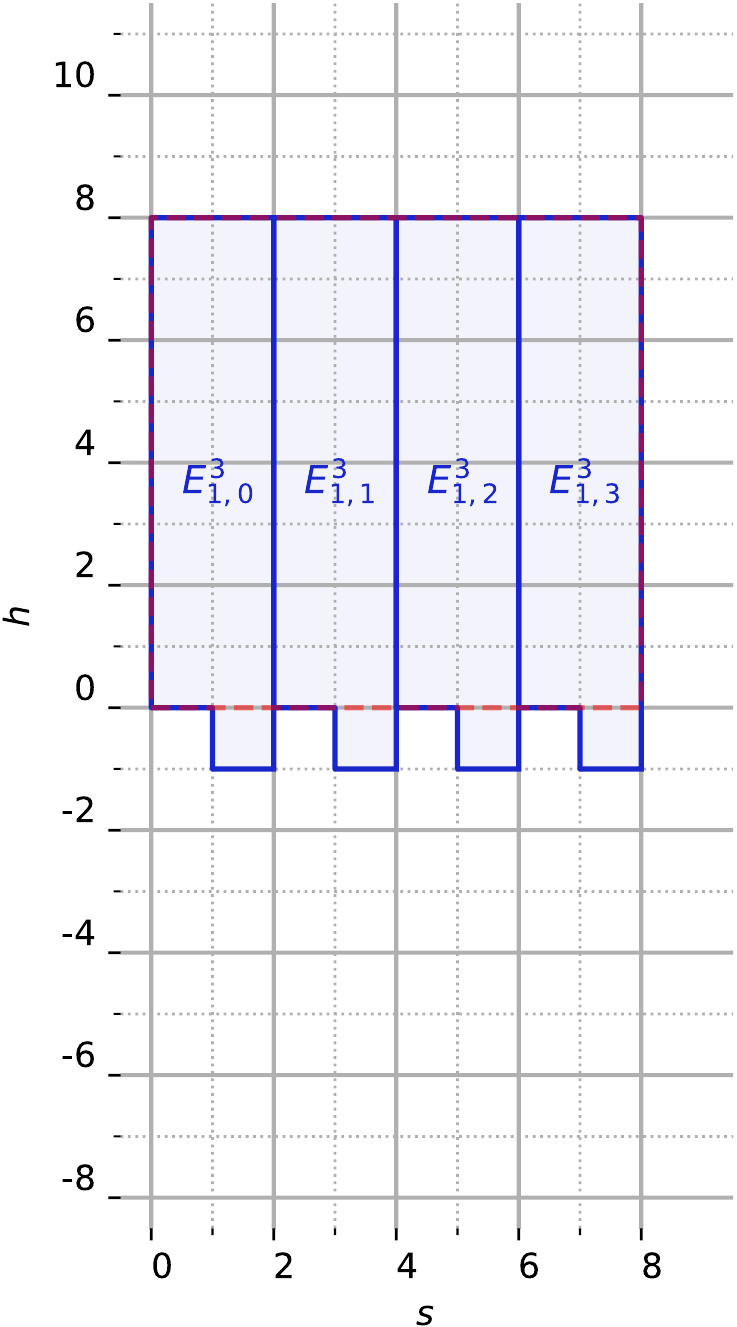}
&
\includegraphics[width=0.3\textwidth]{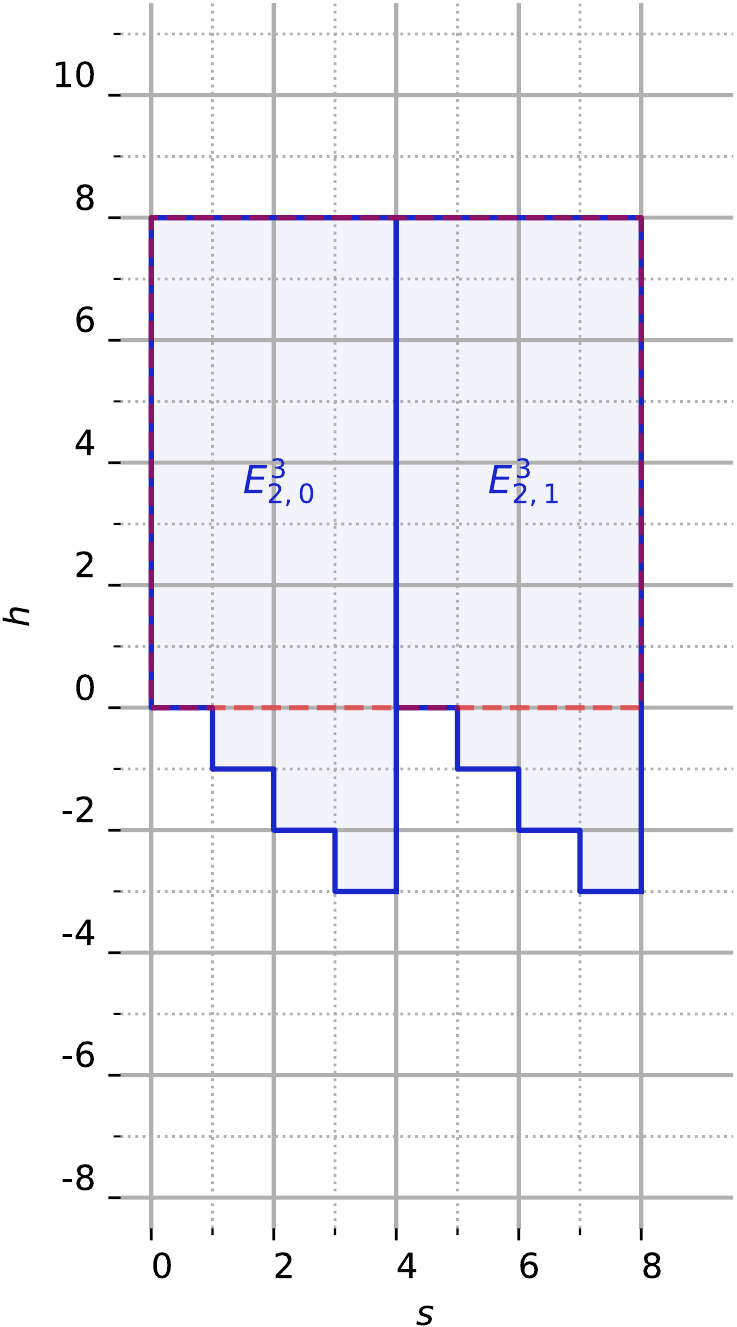}
&
\includegraphics[width=0.3\textwidth]{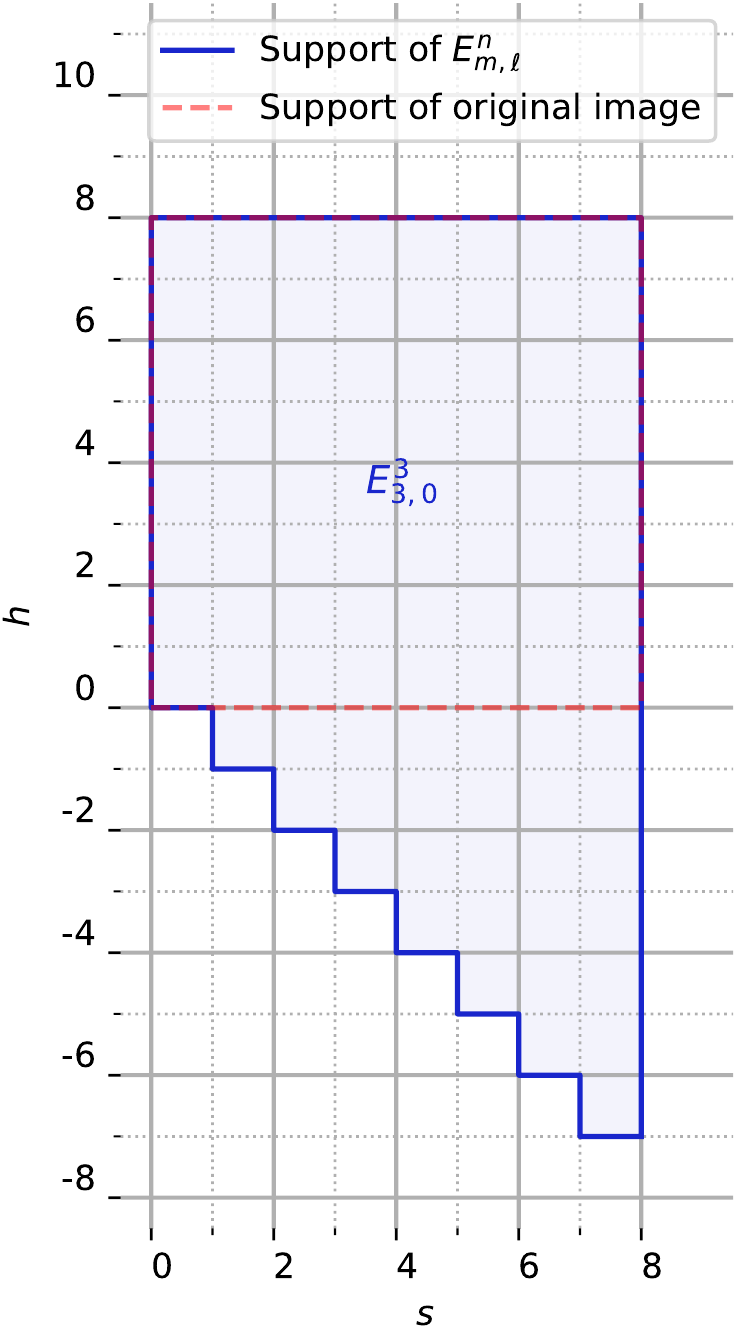}
\end{tabular}
\caption{Diagram depicting the set $E^n_{m,\ell}$ \cref{eq:Enml}.}
\label{fig:Enml}
\end{figure}
 
The following lemma states that if $f \in \cF^n_0$ then $\supp R^n_{m,\ell}[f] \subset E^n_{m,\ell}$, so that the support of $R^n_{m,\ell}[f]$ is restricted in a specific manner.

\begin{lemma}[{Support of $R^n_m[f]$ for $f \in \cF^n_0$}] \label{lem:suppRnmF0}
~  \vspace{-0.1cm}
\begin{enumerate}[label=(\roman*)]
\item If $g \in \cF^n$ with $\supp g \subset \nmlz{E}{m}{2\ell} \cup \nmlz{E}{m}{2\ell+1}$ then $\nm{S}[g] \subset \nmlz{E}{m+1}{\ell}$. 
\item If $f \in \cF^n_0$, it holds that $\supp R^n_m[f] \subset E^n_m$.
\end{enumerate}
\end{lemma} 

\begin{proof}
\begin{enumerate}[label=(\roman*)]
\item By the definition of $S^n_m$,
\[
    \begin{aligned}
    S^n_m [g]
    &=
    \sum_{(h,s + (2\ell) 2^m) \in \nmlz{E}{m}{2\ell}}
    g(h,s) \delta^n_{h,2s+1 + \ell 2^{m+1}}
    \\&\quad +
    \sum_{(h,s+(2\ell+1) 2^m) \in \nmlz{E}{m}{2\ell+1}}
    g(h+s+1,s) \delta^n_{h,2s+1 + \ell 2^{m+1}}
    \\&\quad +
    \sum_{(h,s+(2\ell)2^m) \in \nmlz{E}{m}{2\ell}}
    g(h,s) \delta^n_{h,2s + \ell 2^{m+1}}
    \\&\quad +
    \sum_{(h+s,s+(2\ell+1) 2^m) \in \nmlz{E}{m}{2\ell+1}}
    g(h+s,s) \delta^n_{h,2s + \ell 2^{m+1}}.
    \end{aligned}
\]
Then looking into the first two terms,
\[
    g_0
    :=
    \sum_{s=0}^{2^{m+1}-1}
    \sum_{h=-s}^{2^n-1}
    \left[
        g(h, s + (2\ell) 2^m) + g(h+s, s+(2\ell+1)2^m) 
    \right]
    \delta^n_{h, 2s + 1 + \ell 2^m},
\]
therefore $\supp g_0 \subset \nmlo{E}$. A similar calculation for the latter two terms which we let $g_1$ shows that $\supp g_1 \subset \nmlo{E}$. 
\item Follows from (i), since $\nm{R}[f]$ is just repeated compositions of $S^n_m$.
\end{enumerate}
\end{proof}

Owing to the two types of sums that describe $S_m$ in \cref{eq:Sm}, the map $S^n_m[g]$ for $g \in \cF^n$ and $\supp g \in E^n_m$ can also be written as a linear sum of two types of images, which we now describe. Let us define the functions for each $q \in I_{2^{m-1}}$,
\beq
    \begin{aligned}
    \Phi^n_{m,\ell,q}
    &:=
    \{\phi^n_{m,\ell,p,q} \}_{p= - 2q}^{2^n - 1},
    &
    \phi^n_{m,\ell,p,q} 
    &:=
    \delta_{p,2q +\ell 2^m}^n + \delta^n_{p,2q+1 + \ell 2^m}.
    \\
    \Psi^n_{m,\ell,q}
    &:=
    \{\psi^n_{m,\ell,p,q} \}_{p=-2q}^{2^n - 1},
    &
    \psi^n_{m,\ell,p,q}
    &:=
    \delta_{p,2q+\ell 2^m}^n + \delta^n_{p-1,2q+1+\ell 2^m}.
    \end{aligned}
    \label{eq:PhiPsi}
\eeq
Then define their collections, for $\ell \in I_{2^{n-m}}$,
\beq
    \begin{aligned}
    \Phi^n_{m,\ell} &:= \bigcup_{q \in I_{2^{m-1}}} 
                       \Phi^n_{m,\ell,q},
    &
    \Phi^n_m &:= \bigcup_{\ell \in I_{2^{n-m}}} \Phi^n_{m,\ell},
    \\
    \Psi^n_{m,\ell} &:= \bigcup_{q \in I_{2^{m-1}}}
                       \Psi^n_{m,\ell,q},
    &
    \Psi^n_m &:= \bigcup_{\ell \in I_{2^{n-m}}} \Psi^n_{m,\ell}.
    \end{aligned}
\eeq
Then we have $\supp \nml{\Phi}{\ell}, \supp \nml{\Psi}{\ell} \in
\nml{E}{\ell}$, where $\supp \Phi := \bigcup_{\phi \in \Phi} \supp \phi$.
 
\begin{lemma}[Linear independence and orthogonal relations] 
\label{lem:linindep}~
\begin{enumerate}[label=(\roman*)]
\item $\Phi^n_m$ and $\Psi^n_m$ are each orthogonal individually.  
\item $\Phi^n_{m,\ell,q} \perp \Psi^n_{m,\ell,q'}$, whenever $q \ne q'$.
\item $\Phi^n_{m,\ell,q} \cup \Psi^n_{m,\ell,q}$ is a linearly independent set.
\end{enumerate}
\end{lemma} 
\begin{proof} ~
    \begin{enumerate}[label=(\roman*)]
    \item Whenever $(\ell,p,q) \ne (\ell',p',q')$, we have $\supp \phi^n_{m,\ell,p,q}$ and $\supp \phi^n_{m,\ell',p',q'}$ do not intersect, so that $\phi^n_{m,\ell,p,q} \cdot \phi^n_{m,\ell',p',q'} = 0$. So $\Phi^n_{m}$ is orthogonal, and by the same argument $\Psi^n_m$ is also.
    \item Since $\supp \phi^n_{m,\ell,p,q}, \supp \psi^n_{m,\ell,p,q} \subset \ZZ \times \{ 2p + \ell 2^m, 2p + 1 + \ell 2^m\}$, it follows that $\supp \phi^n_{m,\ell,p,q} \cap \supp \psi^n_{m,\ell',p',q'} = \emptyset$ whenever $q \ne q'$. So the result follows.
    \item Due to (i) and (ii), it only remains to show that $\Phi^n_{m,\ell} \cup \Psi^n_{m,\ell}$ form a linearly independent set. Let $J = 2(2^n+(2q \brem 2^{m-1}))$. We form the real matrix $\bfA \in \RR^{J \times J}$ with entries given by 
\beq
    \bfA_{i,j} 
    = \xi_i \cdot \xi_j,
    \quad i,j \in I_{J},
    \where
    \left\{
    \begin{aligned}
    \xi_{2j} &= \phi^n_{m,\ell,p,j},\\
    \xi_{2j+1} &= \psi^n_{m,\ell,p,j},
    \end{aligned}
    \right.
    \quad
    j \in I_{J/2}.
\eeq
It is straightforward to see that $\bfA_{i,j} = 2$ when $i=j$, and $\bfA_{i,j} = 1$ when $|i-j| = 1$. So $\bfA$ is diagonally dominant and therefore non-singular, showing that $\{\xi_j\}_{j=1}^{J} = \Phi^n_{m,\ell,q} \cup \Psi^n_{m,\ell,q}$ is linearly independent.
    \end{enumerate}
\end{proof}
\begin{figure}
\centering
\def\svgwidth{0.9\textwidth}
{\scriptsize 
\begingroup%
  \makeatletter%
  \providecommand\color[2][]{%
    \errmessage{(Inkscape) Color is used for the text in Inkscape, but the package 'color.sty' is not loaded}%
    \renewcommand\color[2][]{}%
  }%
  \providecommand\transparent[1]{%
    \errmessage{(Inkscape) Transparency is used (non-zero) for the text in Inkscape, but the package 'transparent.sty' is not loaded}%
    \renewcommand\transparent[1]{}%
  }%
  \providecommand\rotatebox[2]{#2}%
  \newcommand*\fsize{\dimexpr\f@size pt\relax}%
  \newcommand*\lineheight[1]{\fontsize{\fsize}{#1\fsize}\selectfont}%
  \ifx\svgwidth\undefined%
    \setlength{\unitlength}{637.74163035bp}%
    \ifx\svgscale\undefined%
      \relax%
    \else%
      \setlength{\unitlength}{\unitlength * \real{\svgscale}}%
    \fi%
  \else%
    \setlength{\unitlength}{\svgwidth}%
  \fi%
  \global\let\svgwidth\undefined%
  \global\let\svgscale\undefined%
  \makeatother%
  \begin{picture}(1,0.63223095)%
    \lineheight{1}%
    \setlength\tabcolsep{0pt}%
    \put(0,0){\includegraphics[width=\unitlength,page=1]{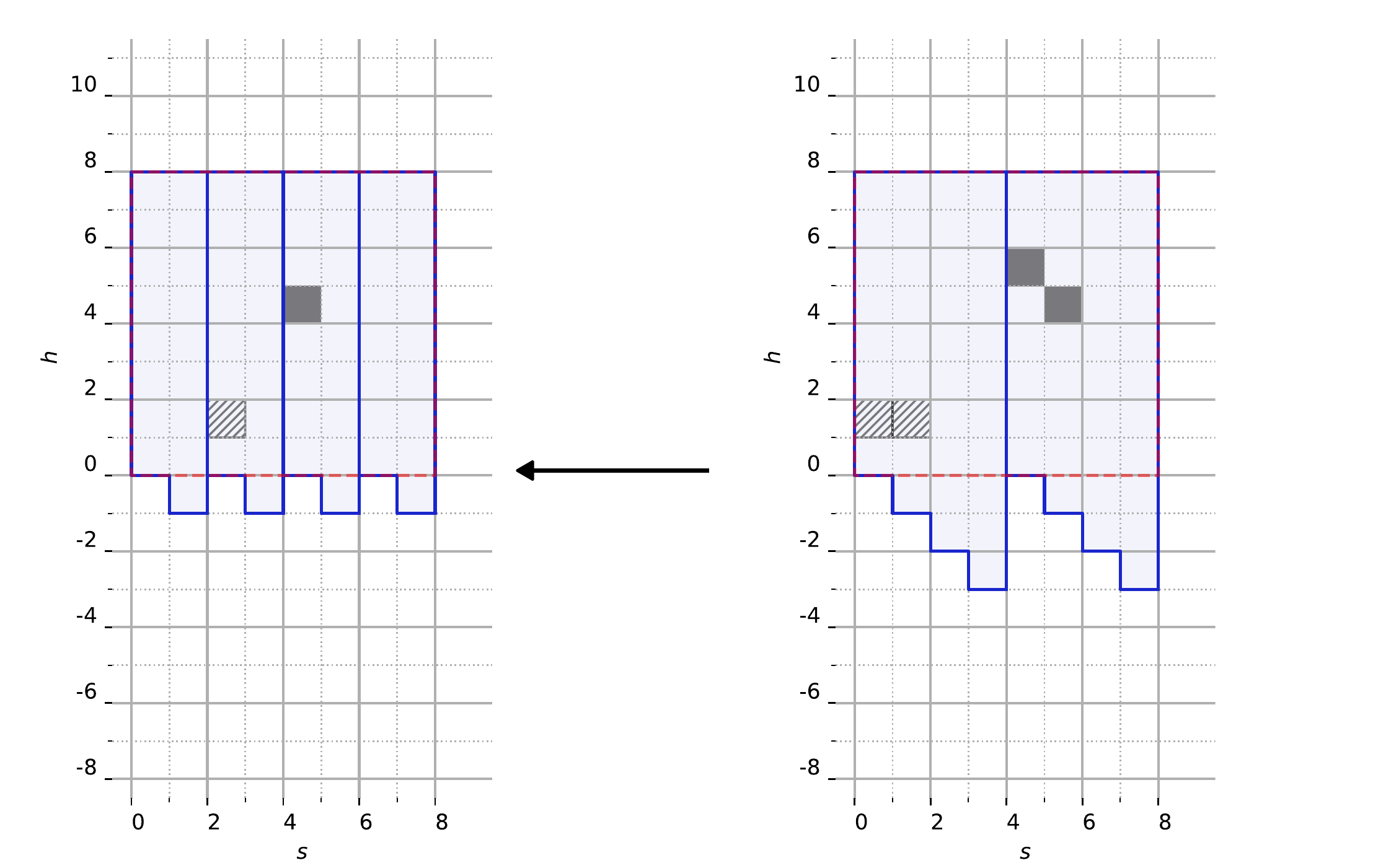}}%
    \put(0.42909718,0.2556857){\color[rgb]{0,0,0}\makebox(0,0)[lt]{\lineheight{1.25}\smash{\begin{tabular}[t]{l}$(S^3_2)^{-1}$\end{tabular}}}}%
    \put(0,0){\includegraphics[width=\unitlength,page=2]{localize-diagram.pdf}}%
    \put(0.42909718,0.32154309){\color[rgb]{0,0,0}\makebox(0,0)[lt]{\lineheight{1.25}\smash{\begin{tabular}[t]{l}$S^3_2$\end{tabular}}}}%
    \put(0.75837783,0.46359732){\color[rgb]{0,0,0.66666667}\makebox(0,0)[lt]{\lineheight{1.25}\smash{\begin{tabular}[t]{l}$\psi^3_{2,1,6,0}$\end{tabular}}}}%
    \put(0.64300191,0.35251958){\color[rgb]{1,0.4,0}\makebox(0,0)[lt]{\lineheight{1.25}\smash{\begin{tabular}[t]{l}$\phi^3_{2,0,1,0}$\end{tabular}}}}%
    \put(0,0){\includegraphics[width=\unitlength,page=3]{localize-diagram.pdf}}%
  \end{picture}%
\endgroup%
}
\caption{A diagram depicting localization (\cref{lem:localize})}
\end{figure}

We will soon see that some of the images in $\nmlo{\Phi}$, $\nmlo{\Psi}$ are not necessary to express $S^n_m[g]$ for $g \in \cF^n$ with $\supp g \subset E^n_{m-1}$. To this end, let us partition $\Phi^n_{m,\ell,q}$ and $\Psi^n_{m,\ell,q}$ each into two sets. For $q \in I_{2^{m-1}}$,
\beq \label{eq:huPhiPsi}
    \begin{aligned}
    \oPhi^n_{m,\ell,q} &:= \{\phi^n_{m,\ell,p,q}\}_{p=-q}^{2^n-1},
    &
    \oPsi^n_{m,\ell,q} &:= \{\psi^n_{m,\ell,p,q}\}_{p=-2q}^{2^n-q-1},
    \\
    \uPhi^n_{m,\ell,q} &:= \{\phi^n_{m,\ell,p,q}\}_{p=-2q}^{-q-1},
    &
    \uPsi^n_{m,\ell,q} &:= \{\psi^n_{m,\ell,p,q}\}_{p=2^n-q}^{2^n}.
    \end{aligned}
\eeq
An illustration of the members in $\uPhi^n_m$ and $\uPsi^n_m$ are shown in \cref{fig:uphi-upsi}. 

Similarly as before, we let
\beq
\begin{aligned}
    \oPhi^n_{m,\ell} 
    &:= 
    \bigcup_{q \in I_{2^{m-1}}} \oPhi^n_{m,\ell,q},
    &
    \oPsi^n_{m,\ell} 
    &:= 
    \bigcup_{q \in I_{2^{m-1}}} \oPsi^n_{m,\ell,q},
    \\
    \uPhi^n_{m,\ell} 
    &:= 
    \bigcup_{q \in I_{2^{m-1}}} \uPhi^n_{m,\ell,q},
    &
    \uPsi^n_{m,\ell} 
    &:= 
    \bigcup_{q \in I_{2^{m-1}}} \uPsi^n_{m,\ell,q},
    \\
    \oPhi^n_m 
    &:= 
    \bigcup_{\ell \in I_{2^{m-1}}} \oPhi^n_{m,\ell},
    &
    \oPsi^n_m 
    &:= 
    \bigcup_{\ell \in I_{2^{m-1}}} \oPsi^n_{m,\ell}.
    \\
    \uPhi^n_m 
    &:= 
    \bigcup_{\ell \in I_{2^{m-1}}} \uPhi^n_{m,\ell},
    &
    \uPsi^n_m 
    &:= 
    \bigcup_{\ell \in I_{2^{m-1}}} \uPsi^n_{m,\ell}.
\end{aligned}
\eeq
Let $\oPhi^n := \oPhi^n_n, \oPsi^n := \oPsi^n_n, \uPhi^n := \uPhi^n_n, \uPsi^n := \uPsi^n_n$ as well.

For an image $g \in \cF^n$, if $S^n_m[g]$ belongs to $\Span (\nmlo{\oPhi} \cup \nmlo{\oPsi})$, one can deduce whether the support of $f$ overlaps with $\nml{E}{2\ell}$ or $\nml{E}{2\ell+1}$ by writing $S^n_m[g]$ as a linear sum of members of $\nmlo{\Phi} \cup \nmlo{\Psi}$ and checking if the coefficients corresponding to those of $\nmlo{\Phi}$ or $\nmlo{\Psi}$ vanish.

\begin{lemma}[Localization] \label{lem:localize} 
Let $g = S^n_m [g_0 + g_1]$ where $g_0,g_1 \in \cF^n$, $\supp g_0 \in E^n_{m-1,2\ell}, \supp g_1 \in
E^n_{m-1,2\ell+1}$ then $g \in \Span (\nmlo{\oPhi} \cup \nmlo{\oPsi})$. And furthermore,
\begin{enumerate}[label=(\roman*)]
\item if $g \in \Span \nmlo{\oPhi}$ then $g_1 = 0$,
\item if $g \in \Span \nmlo{\oPsi}$ then $g_0 = 0$.
\end{enumerate}
\end{lemma} 

\begin{proof}
Let us write,
\beq
    g_0 = \sum_{(h,s) \in E^m} g_0(h,s) \delta^n_{h,s},
    \qquad
    g_1 = \sum_{(h,s) \in E^m} g_1(h,s) \delta^n_{h,s}.
\eeq
Then by \cref{eq:Sm}, we have
\beq
    \begin{aligned}
    g
    &=
    \sum_{(h,s) \in \nmlz{E}{m}{2\ell}}
    g_0(h,s) \delta^n_{h,2s+1}
    +
    \sum_{(h+s+1,s) \in \nmlz{E}{m}{2\ell+1}}
    g_1(h+s+1,s) \delta^n_{h,2s+1}
    \\&\qquad +
    \sum_{(h,s) \in \nmlz{E}{m}{2\ell}}
    g_0(h,s) \delta^n_{h,2s}
    +
    \sum_{(h+s,s) \in \nmlz{E}{m}{2\ell+1}}
    g_1(h+s,s) \delta^n_{h,2s}
    \\&=
    \sum_{(h,s) \in \nmlz{E}{m}{2\ell}}
    g_0(h,s) (\delta^n_{h,2s} + \delta^n_{h,2s+1})
    \\&\qquad +
    \sum_{(h,s) \in \nmlz{E}{m}{2\ell+1}}
    g_1(h,s) (\delta^n_{h-s,2s} + \delta^n_{h-s-1,2s+1})
    \\&=
    \sum_{s = 0}^{2^m-1} \sum_{h=-s}^{2^n-1}
    g_0(h,s) (\delta^n_{h,2s} + \delta^n_{h,2s+1})
    \\&\qquad +
    \sum_{s = 0}^{2^m-1} \sum_{h=-s}^{2^n-1}
    g_1(h,s) ( \delta^n_{h-s,2s} + \delta^n_{h-s-1,2s+1}),
    \end{aligned}
\eeq
then using the definitions of the basis $\psi$ and $\phi$ \cref{eq:PhiPsi},
\beq
    g
    =
    \sum_{q = 0}^{2^m-1} \sum_{p=-q}^{2^n-1}
    g_0(p,q) \nmlpqz{\phi}{p}{q}
    +
    \sum_{q = 0}^{2^m-1} \sum_{p=-q}^{2^n-1}
    g_1(p,q) \nmlpqz{\psi}{p-q}{q}.
\eeq
Rearranging, we arrive at
\beq
    g
    =
    \sum_{q = 0}^{2^m-1} \sum_{p=-q}^{2^n-1}
    g_0(p,q) \nmlpqz{\phi}{p}{q}
    +
    \sum_{q = 0}^{2^m-1} \sum_{p=-2q}^{2^n-q-1}
    g_1(p,q) \nmlpqz{\psi}{p}{q},
    \label{eq:Fpshi}
\eeq
showing that $g \in \Span (\nmlo{\oPhi} \cup \nmlo{\oPsi})$.  Recall that the set $\nmlo{\oPhi} \cup \nmlo{\oPsi}$ is linearly independent, so the expansion \cref{eq:Fpshi} is unique. Hence, if $f \in \Span \nmlo{\oPhi}$ it must be that $f_0 = 0$. Similarly, if $f \in \Span \nmlo{\oPsi}$ then $f_1 = 0$.
\end{proof}

Note that the inversion formula \cref{eq:inv-formula} is a backward substitution formula applied to the linear system \cref{eq:Fpshi}.

Now, we turn our attention to the space of images $f \in \cF^n$ whose support lies in $E^n_{m,\ell}$ defined for $m \in I_{n+1}$
\beq
    \cG^{n}_{m,\ell} 
    := 
    \{ f \in \cF^n : \supp f \in E^n_{m,\ell} \},
    \quad
    \cG^n_m 
    := 
    \bigcup_{\ell \in I_{2^{n-m}}} \cG^n_{m,\ell}.
    \label{eq:hOmega}
\eeq
Above, we have shown $g \in \cG^n_{m,\ell}$ such that $g = S^n_m[g_0 + g_1]$ for some $g_0 \in \cG^n_{m-1,2\ell}$, $g_1 \in \cG^n_{m-1,2\ell+1}$ to belong to $\Span (\nmlo{\oPhi} \cup \nmlo{\oPsi})$, and therefore to $\Span (\nmlo{\Phi} \cup \nmlo{\Psi})$. Then, it is natural to ask which of those $f \in \cG^n_{m,\ell}$ do \emph{not} belong to $\Span (\nmlo{\Phi} \cup \nmlo{\Psi})$? We answer this question next. 

Let us define $\nmlq{\mu} \in \cF^n$ given by
\beq
    \begin{aligned}
    \nmlq{\mu}(h,s)
    &=
    \begin{cases} 
    1 & \text{ if } -2q \le h \le 2^n - 1 \text{ and } s = 2q + \ell 2^m. \\
    -1 & \text{ if } -2q-1\le h \le 2^n - 1 \text{ and } s = 2q + 1 + \ell 2^m,\\
    0 & \text{ otherwise.}
    \end{cases}\\
    &= 
    \sum_{p = -2q}^{2^n-1} 
    \delta^n_{p,2q + \ell 2^m} (h,s)
    -
    \sum_{p = -2q-1}^{2^n-1} 
    \delta^n_{p,2q + 1 + \ell 2^m}
    \end{aligned}
\eeq
and correspondingly let 
\beq \label{eq:Mnml}
\nml{M}{\ell} := \{ \nmlq{\mu} \}_{q \in I_{2^{m-1}}}. 
\eeq
Note that $\nmlo{M}$ forms an orthogonal set, as $\supp \mu^n_{m,\ell,q} \cap \supp \mu^n_{m,\ell,q'} = \emptyset$ for $q \ne q'$.

\begin{figure}
\centering
    \begin{tabular}{ccc}
    \includegraphics[width=0.3\textwidth]{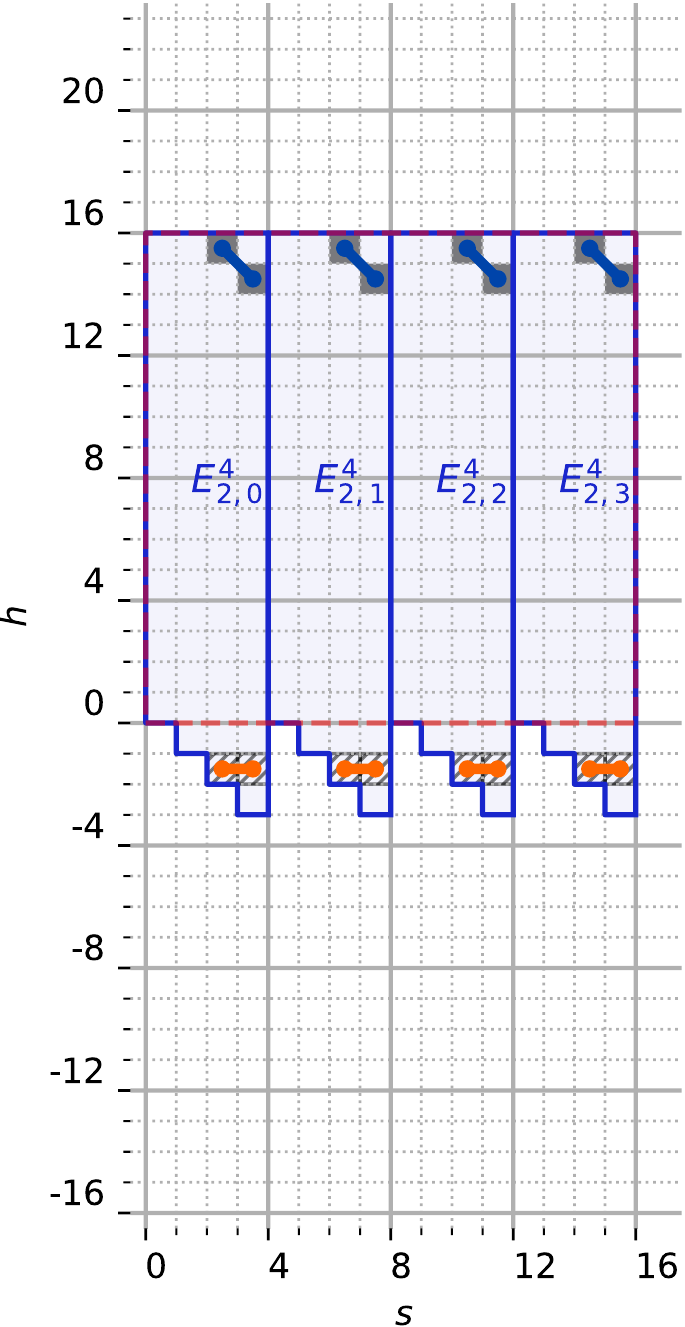}
    &
    \includegraphics[width=0.3\textwidth]{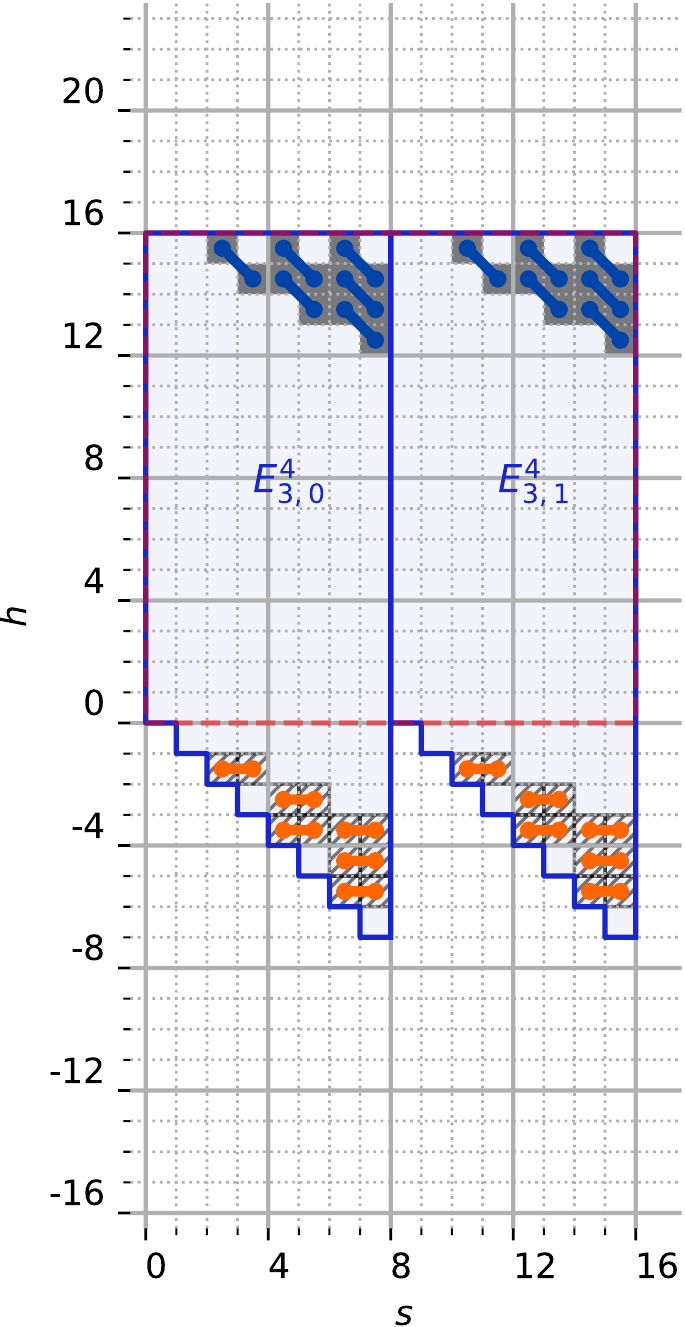}
    &
    \includegraphics[width=0.3\textwidth]{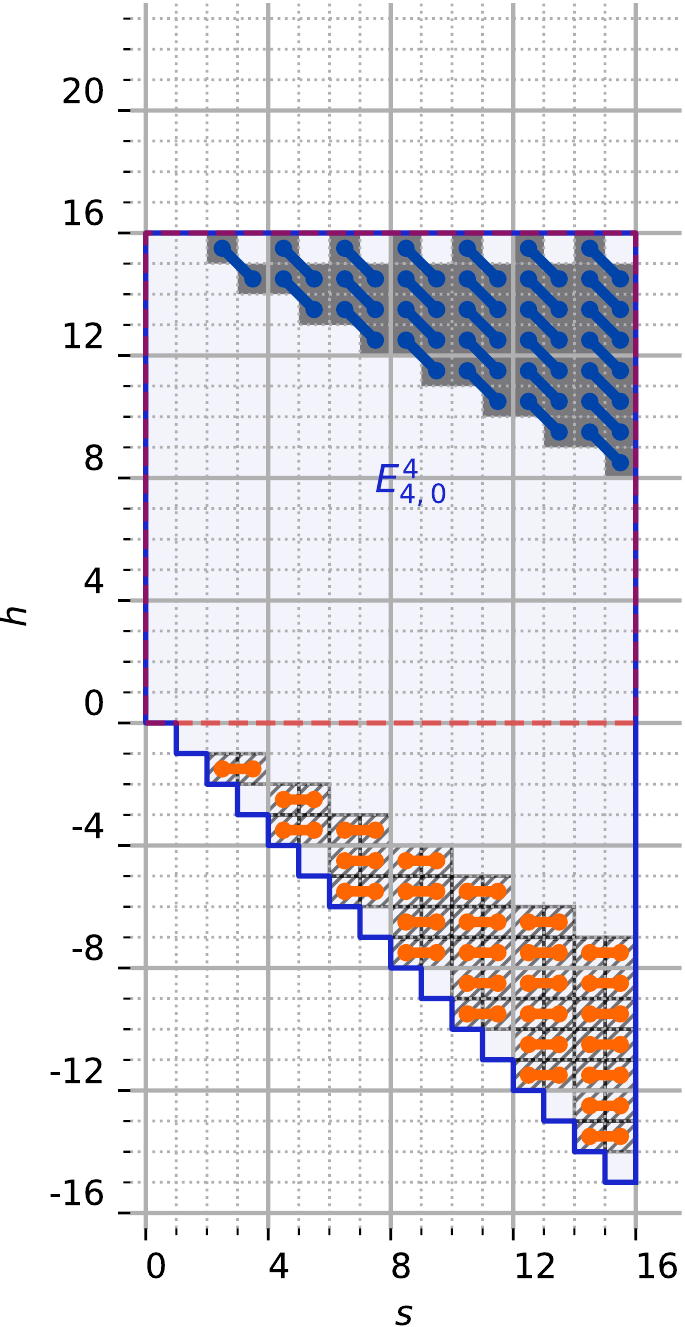}
    \end{tabular}
\caption{An illustration of members of $\uPhi^n_m$ and $\uPsi^n_m$ for $n=4$ with $m=2$ (left), $m=3$ (middle) and $m=4$ (right).}
\label{fig:uphi-upsi}
\end{figure}

\newpage

\begin{lemma}[Mass constraints] \label{lem:mass}
Suppose $g \in \nmlo{\cG}$, then $g \in \Span (\nmlo{\Phi} \cup \nmlo{\Psi})$ if and only if $g \cdot \mu = 0$ for all $\mu \in \nml{M}{\ell}$. Equivalently,
\beq
    \sum_{h=-2t}^{2^n-1} g(h,2t + \ell 2^m)
    =
    \sum_{h=-2t-1}^{2^n-1} g(h,2t+1 + \ell 2^m)
    \quad
    \text{ for } t \in I_{2^{m-1}}.
    \label{eq:mass}
\eeq
\end{lemma} 
\begin{proof}
The members of $\nmlo{\Phi}$ and $\nmlo{\Psi}$ themselves satisfy this constraint: for $q \in I_{2^{m-1}}$ 
\[
\begin{aligned}
\sum_{h=-2t}^{2^n} \nml{\phi}{\ell,p,q}(h,2t) 
&= 
\sum_{h=-2t-1}^{2^n} \nml{\phi}{\ell,p,q} (h,2t+1),
\\
\sum_{h=-2t}^{2^n} \nml{\psi}{\ell,p,q} (h,2t) 
&= 
\sum_{h=-2t-1}^{2^n} \nml{\psi}{\ell,p,q} (h,2t+1).
\end{aligned}
\]
That is, $\mu \cdot \psi =0$ for all $\mu \in \nml{M}{\ell}$ and
$\psi \in \nml{\Phi}{\ell} \cup \nml{\Psi}{\ell}$. So if $g \in
\Span (\nmlo{\Phi} \cup \nmlo{\Psi})$ then $g \cdot \mu =0$. 

Conversely, $\nmlq{E} = \nml{E}{\ell} \cap (\ZZ
\times \{2q, 2q+1\})$ ($q \in I_{2^{m-1}}$) has $2\cdot 2^n + 4q + 1$ members,
whereas $|\nml{\Phi}{\ell,q}| + |\nml{\Psi}{\ell,q}| = (2^n + 2q) + (2^n + 2q) =
2\cdot2^n + 4q$. 
As noted above, $\nml{\mu}{\ell,p,q} \cdot \psi =0$ for all $\psi \in \nml{\Phi}{\ell,q} \cup \nml{\Psi}{\ell,q}$. Hence $\{\nml{\mu}{\ell,p,q}\} \cup \nml{\Phi}{\ell,q} \cup \nml{\Psi}{\ell,q}$ form a basis for functions in $\cF^n$ supported in $\nml{E}{\ell,q}$. Therefore $g \cdot \nml{\mu}{\ell,p,q} = 0$ implies $g \in \Span (\nml{\Phi}{\ell,q} \cup \nml{\Psi}{\ell,q})$ for all $q \in I_{2^{m-1}}$, implying that $g \in \Span (\nmlo{\Phi} \cup \nmlo{\Psi})$.
\end{proof}

The constraint \cref{eq:mass} is interpreted as a sort of a mass consistency condition. If $g = \nmlo{R}[f]$ for some $f \in \cF^n_0$, then the constraint implies that
\beq
    \sum_{h=-2t}^{2^n-1} \nmlo{R}[f](h,2t)
    =
    \sum_{h=-2t-1}^{2^n-1} \nmlo{R}[f](h,2t+1).
\eeq
This condition is always satisfied for such $g$ since the sum of $f$ over digital lines of different slopes $2t$ and $2t+1$,
\[
    \bigcup_{h=-2t}^{2^n-1} \nmlo{D} \brkt{h}{2t},
    \Aand
    \bigcup_{h=-2t-1}^{2^n-1} \nmlo{D} \brkt{h}{2t+1},
\]
must both equal the total sum  $\sum_{(i,j) \in \ZZ \times I_{2^m}} f^n_{m,\ell}(i,j)$ of all values of $f$.

\subsection{Characterization of $R^n[\cF^n_0]$} 

Now we characterize the range of $\cF^n_0$ under $R^n$. First, note that the range of $S^n_m$ is straightforward to characterize, due to the lemmas from the preceding section.

First, it follows from \cref{lem:linindep} and \cref{lem:mass} that
\beq
    \begin{aligned}
    \cG^{n}_{m,\ell}
    &=
    \Span (
    \oPhi^n_{m,\ell}
    \cup 
    \oPsi^n_{m,\ell}
    \cup 
    \uPhi^n_{m,\ell}
    \cup
    \uPsi^n_{m,\ell}
    \cup
    M^n_{m,\ell}),
    \label{eq:direct-sum}
    \end{aligned}
\eeq
in which $\nmlo{\oPhi}, \nmlo{\oPsi}, \nmlo{\uPhi},\nmlo{\uPsi}$ defined in \cref{eq:huPhiPsi}, and $\nmlo{M}$ in \cref{eq:Mnml}. 

We will separate $\cG^n_{m,\ell}$ in two, by writing  
\beq \label{eq:Gnml-split}
\cG^n_{m,\ell} = \Span X^n_{m,\ell} \oplus \Span Y^n_{m,\ell},
\eeq
where $\oplus$ denotes the direct sum and
\beq
    X^n_{m,\ell} 
    :=
    \oPhi^n_{m,\ell}
    \cup
    \oPsi^n_{m,\ell},
    \qquad
    Y^n_{m,\ell} 
    :=
    \uPhi^n_{m,\ell}
    \cup 
    \uPsi^n_{m,\ell}
    \cup 
    M^n_{m,\ell}.
\eeq
We will also let,
\[
    \nm{X} := \bigcup_{\ell \in I_{2^{n-m}}} \nmlo{X},
    \quad
    \nm{Y} := \bigcup_{\ell \in I_{2^{n-m}}} \nmlo{Y}.
\]
Let $\nml{P}{\ell}$ be a projection to $\Span \nmlo{Y}$, and $\nm{P}$ a projection to $\Span \nm{Y}$.

\begin{lemma} \label{lem:rangeSnm} 
It holds that $\nm{S}[\cG^n_{m-1}] = \Span X^n_{m}$ for $m=1,...,n$.
\end{lemma}

\begin{proof}
Let $g = \nm{S} [f]$ for some $f \in \cG^n_{m-1}$, then $g \in \Span X^n_{m}$ by \cref{lem:localize}, so $\nm{S} [\cG^n_{m-1}] \subset \Span X^n_{m}$. Conversely, if $g \in \Span X^n_m$ then $(\nm{S})^{-1}[g] \in \cG^n_{m-1}$ by \cref{lem:localize}, so $\nm{S}[\cG^n_{m-1}] \supset \Span X^n_m$.
\end{proof}

We will now characterize the range $R^n[\cF^n_0]$. So far, we know from  \cref{lem:suppRnmF0} that $R^n$ maps $\cF_0^n$ into $\cG^n$ \cref{eq:hOmega}, and have shown in \cref{thm:Fp-Fp} that $R^n : \cF^n_+ \to \cF^n_+$ is a bijection by the virtue of the inversion formula \cref{eq:inv-formula+}. Then we showed that $S^n_m [\cG^n_m] = \Span X^n_m$ in \cref{lem:rangeSnm}. These facts will enable us to characterize $R^n[\cF^n_0] \subset \cG^n$ with a set of linear constraints.

We define the partial inverse $T^n_m: \cF^n_+ \to \cF^n_+$ as given by
\beq
    \left\{
    \begin{aligned}
    \nm{T}
    &:= 
    (\nmz{S}{m+1})^{-1} 
    \circ ... \circ
    (\nmz{S}{n})^{-1}
    \text{ for } m = 2, ..., n-1,
    \\
    T^n_n 
    &:= 
    \Id.
    \end{aligned}
    \right.
\eeq

Now, we will define the space of images formed from  $g \in \cG^n$ \cref{eq:hOmega} for which $T^n_m[g]$ has projection $\nm{P}$ to $\Span Y^n_m$ \cref{eq:Gnml-split} is zero.
\beq
    \cF^n_R
    =
    \left\{
        g \in \cG^n
        \,|\,
        \nm{P} \circ T^n_m [g] = 0, \,
        \, m = 2, ... , n
    \right\}.
    \label{eq:range}
\eeq
We will show that $\cF^n_R$ is the range of $\cF^n_0$ under $R^n$.

\begin{theorem}[Range characterization of $R^n$]
$R^n : \cF^n_0 \to \cF^n_R$ is a bijection.
\end{theorem}

\begin{proof}
For $f \in \cF^n_0$, $\supp f \subset I_{2^n} \times I_{2^n}$ so $f$ can be identified to a function from $I_{2^n} \times I_{2^n}$ to $\RR$, hence $f$ has $\abs{I_{2^n} \times I_{2^n}} = 2^{2n} = N^2$ degrees of freedom, and $R^n[f] \subset \cG^n$. Any $g \in \cG^n$ is identified to an image from $E^n$ to $\RR$, so the number of degree of
freedoms for $g$ is
\[
    |E^n| = 2^{2n} + 2^{2n-1} - 2^{n-1} 
    = 
    N^2 + \frac{N(N-1)}{2}
    =
    \frac{3}{2} N^2 - \half N.
\]
To show that $R^n: \cF^n_0 \to \cF^n_R$ is a bijection, it suffices to show that the constraints 
\beq
\nm{P} \circ \nm{T} [g] = 0
\quad \text{ for }
m=2, \, ... \, , n,
\eeq
form $N(N-1)/2$ linearly independent constraints.

Now, consider the identity
\beq
    \begin{aligned}
    \inmzS{m} [ \nm{\cG} ]
    &=
    \inmzS{m} [ \Span \nm{X} ] 
    \oplus
    \inmzS{m} [ \Span \nm{Y} ]
    \\ &=
    \cG ^n_{m-1}
    \oplus
     \inmzS{m} [ \Span \nm{Y} ] .
    \end{aligned}
\eeq
Applying this to $\cG^n$,
\beq
    \begin{aligned}
    (R^n)^{-1}[ \cG^n ]
    &=
    \inmzS{1} \circ \cdots \circ \inmzS{n}
    [ \cG^n ]
    \\ &=
    \inmzS{1} \circ \cdots \circ \inmzS{n}
    [ \Span X^n \oplus \Span Y^n ]
    \\ &=
    \inmzS{1} \circ \cdots \circ \inmzS{n-1}
    \left[
    \cG^n_{n-1} 
    \oplus 
    \inmzS{n}[ \Span Y^n]
    \right].
    \end{aligned}
\eeq
Then repeating, 
\beq
    \begin{aligned}
    & \inmzS{1} \circ \cdots \circ \inmzS{n-1}  
    [ \cG^n_{n-1} \oplus \inmzS{n}[\Span Y^n] ]
    \\ &= 
     \inmzS{1} \circ \cdots \circ \inmzS{n-1}  
    [ \Span \nmz{X}{n-1} 
    \oplus 
    \Span \nmz{Y}{n-1} 
    \oplus
    \inmzS{n}[\Span Y^n]]
    \\ &= 
    ( \inmzS{1} \circ \cdots \circ \inmzS{n-2}  )
    \\ &\qquad \qquad
    \left[ 
       \nmz{\cG}{n-2}  
       \oplus
       \inmzS{n-1} [\Span Y^n_{n-1}] 
       \oplus 
       \inmzS{n-1} \circ \inmzS{n} [\Span Y^n] 
    \right]
    \\ &= \cdots
    =
    \cG^n_0
    \oplus
    \left(
        \bigoplus_{m=1}^{n} (\nm{R})^{-1} [ \Span \nm{Y}] 
    \right).
    \end{aligned}
\eeq
Note that since $\nm{R}$ is a bijection from $\ZZ \times I_{2^n}$ to $\ZZ \times I_{2^n}$, for $g \in \cG^n$, we have that $\nmlo{P} \circ \nm{T} [g] = 0$ if and only if $\nm{P} \circ R^n_m \circ (R^n_m)^{-1} \circ \nm{T} [g] = 0$, and writing $\nm{Q} := \nm{P} \circ R^n_m$, $\nmlo{Q}$ is a projection in $\cF^n$ to the subspace
\[
\Span \nm{Z}
\where
\nm{Z}:=(R^n_m)^{-1} [Y^n_m].
\]
Thus, the constraint $\nm{P} \circ \nm{T} [g] = 0$ is equivalent to $\nm{Q} \circ (R^n)^{-1} [g] = 0$. Now, since $\cG^n_0 = \cF^n_0$,
\beq
    (R^n)^{-1}[\cG^n]
    =
    \cF^n_0
    \oplus
    \left(
        \bigoplus_{m=1}^n 
        \Span \nm{Z} 
    \right),
\eeq
and the projection $\nm{Q}$ is linearly independent spaces $\Span \nm{Z}$, and the set of constraints $\nm{P} \circ \nm{T} [g] = 0$ for $m =2, ... , n$ are linearly independent.

We next count the number of constraints. The  total number $\cC(N)$ of
constraints is composed of the mass constraints $\cC_M(N)$ and the
support constraints $\cC_S(N)$, $\cC(N) = \cC_{M}(N) + \cC_{S}(N).$ Then, we
have $N/2$ mass constraints at each application of $S^n_m$ \cref{eq:Sm} for
$m=1, ... , n$ is $C_M(N) = \frac{N}{2} n$. The number of additional constraints
given by the projections $P^n_{m}$ at each $m = 2, ..., n$ is given by 
\beq
    |\uPhi^n_m| + |\uPsi^n_m|
    =
    2^{m-1} \cdot 
    \left( 2^{m-1} - 1 \right)
    \cdot 2^{n-m} 
    = N\frac{2^{m-1}-1}{2}.
\eeq
Note that when $m=1$ there are no constraints, as $\uPhi^n_1,\uPsi^n_1 = \emptyset$: as $|E^n_1| = N^2 + N/2$ for $\ell \in I_{2^{n-1}}$ with $N/2$ extra degrees of freedom, which matches the number of mass constraints for $m=1$.

So the total number of constraints is
\beq
    \begin{aligned}
    C(N) &=
    \frac{N}{2} n
    +
    N
    \left(
    \sum_{m=2}^{n} \frac{2^{m-1}-1}{2} 
    \right)
    \\ &=
    \frac{N}{2} n
    +
    N
    \left( 
    \sum_{m=2}^{n} 2^{m-2}
    \right)
    -
    N
    \left( \sum_{m=2}^{n} \frac{1}{2} \right)
    =
    \frac{N(N-1)}{2}.
    \end{aligned}
\eeq
The number of linearly independent constraints match the redundant number of degrees of freedom, proving the claim.
\end{proof}

\section*{Acknowledgements}
WL was supported by AMS Simons Travel grant. The work of KR is partially
supported by the National Science Foundation through grants DMS-1913309 and
DMS-1937254.  The work of DR was partially supported by the Air Force Center of
Excellence on Multi-Fidelity Modeling of Rocket Combustor Dynamics under Award
Number FA9550-17-1-0195 and AFOSR MURI on multi-information sources of
multi-physics systems under Award Number FA9550-15-1-0038.

\bibliographystyle{../common/siamplain}

\begin{thebibliography}{10}

    \bibitem{Averbuch08ppft}
    {\sc A.~Averbuch, R.~R. Coifman, D.~L. Donoho, M.~Israeli, and Y.~Shkolnisky},
      {\em A framework for discrete integral transformations {I}—the pseudopolar
      fourier transform}, SIAM Journal on Scientific Computing, 30 (2008),
      pp.~764--784.
    
    \bibitem{Averbuch08drt}
    {\sc A.~Averbuch, R.~R. Coifman, D.~L. Donoho, M.~Israeli, Y.~Shkolnisky, and
      I.~Sedelnikov}, {\em A framework for discrete integral transformations
      {II}—the 2{D} discrete {R}adon transform}, SIAM Journal on Scientific
      Computing, 30 (2008), pp.~785--803.
    
    \bibitem{Beylkin87drt}
    {\sc G.~{Beylkin}}, {\em Discrete radon transform}, IEEE Transactions on
      Acoustics, Speech, and Signal Processing, 35 (1987), pp.~162--172.
    
    \bibitem{Bonneel15slice}
    {\sc N.~Bonneel, J.~Rabin, G.~Peyr\'{e}, and H.~Pfister}, {\em Sliced and
      {R}adon {W}asserstein barycenters of measures}, Journal of Mathematical
      Imaging and Vision, 51 (2015), pp.~22--45.
    
    \bibitem{Brady98adrt}
    {\sc M.~L. Brady}, {\em A fast discrete approximation algorithm for the {R}adon
      transform}, SIAM Journal on Computing, 27 (1998), pp.~107--119.
    
    \bibitem{Cormack63}
    {\sc A.~M. Cormack}, {\em {Representation of a Function by Its Line Integrals,
      with Some Radiological Applications}}, Journal of Applied Physics, 34 (1963),
      pp.~2722--2727.
    
    \bibitem{Dean83}
    {\sc S.~R. Deans}, {\em {The Radon Transform and Some of Its Applications}},
      Wiley, 1983.
    
    \bibitem{Frank1996}
    {\sc J.~Frank}, {\em {Three-Dimensional Electron Microscopy of Macromolecular
      Assemblies}}, Academic Press, Burlington, 1996.
    
    \bibitem{Gotz96fdrt}
    {\sc W.~G\"otz and H.~Druckm\"uller}, {\em A fast digital radon transform: an
      efficient means for evaluating the hough transform}, Pattern Recognition, 29
      (1996), pp.~711 -- 718.
    
    \bibitem{Helgason99}
    {\sc S.~Helgason}, {\em The Radon Transform}, Springer, Boston, MA., 1999.
    
    \bibitem{Ilmavirta2015}
    {\sc J.~Ilmavirta}, {\em {On Radon transforms on tori}}, Journal of Fourier
      Analysis and Applications, 21 (2015), pp.~370--382.
    
    \bibitem{Ilmavirta20}
    {\sc J.~Ilmavirta, O.~Koskela, and J.~Railo}, {\em Torus computed tomography},
      SIAM Journal on Applied Mathematics, 80 (2020), pp.~1947--1976.
    
    \bibitem{IlmavirtaMonard2019}
    {\sc J.~Ilmavirta and F.~Monard}, {\em 4. {I}ntegral geometry on manifolds with
      boundary and applications}, De Gruyter, 2019, pp.~43--114.
    
    \bibitem{Ilmavirta2018}
    {\sc J.~Ilmavirta and G.~Uhlmann}, {\em Tensor tomography in periodic slabs},
      Journal of Functional Analysis, 275 (2018), pp.~288--299.
    
    \bibitem{KazantsevBukhgeim04}
    {\sc S.~G. Kazantsev and A.~A. Bukhgeim}, {\em Singular value decomposition for
      the 2d fan-beam radon transform of tensor fields}, Journal of Inverse and
      Ill-Posed Problems, 12 (2004), pp.~245--278.
    
    \bibitem{Kelley93frt}
    {\sc B.~T. {Kelley} and V.~K. {Madisetti}}, {\em The fast discrete {R}adon
      transform. {I}. {T}heory}, IEEE Transactions on Image Processing, 2 (1993),
      pp.~382--400.
    
    \bibitem{Lax64scattering}
    {\sc P.~D. Lax and R.~S. Phillips}, {\em Scattering theory}, Bull. Amer. Math.
      Soc., 70 (1964), pp.~130--142.
    
    \bibitem{Louis84}
    {\sc A.~K. Louis}, {\em {Orthogonal Function Series Expansions and the Null
      Space of the Radon Transform}}, SIAM Journal on Mathematical Analysis, 15
      (1984), pp.~621--633.
    
    \bibitem{Louis85}
    {\sc A.~K. Louis}, {\em {Tikhonov-Phillips Regularization of the Radon
      Transform}}, in International Series of Numerical Mathematics, vol 73,
      G.~H\"{a}mmerlin and K.~H. KH, eds., Birkh\"{a}user, Basel, 1985,
      ch.~International Series of Numerical Mathematics, vol 73.
    
    \bibitem{LouisRieder89}
    {\sc A.~K. Louis and A.~Rieder}, {\em {Incomplete Data Problems in X-Ray
      Computerized Tomography}}, Numer. Math., 56 (1989), pp.~371--383.
    
    \bibitem{Louis96}
    {\sc A.~K. Louis and T.~Schuster}, {\em A novel filter design technique in 2{D}
      computerized tomography}, Inverse Problems, 12 (1996), pp.~685--696.
    
    \bibitem{Maass87}
    {\sc P.~Maass}, {\em The x-ray transform: singular value decomposition and
      resolution}, Inverse Problems, 3 (1987), pp.~729--741.
    
    \bibitem{Maass92}
    {\sc P.~Maass}, {\em {The Interior Radon Transform}}, SIAM Journal on Applied
      Mathematics, 52 (1992), pp.~710--724.
    
    \bibitem{Marr1974OnTR}
    {\sc R.~B. Marr}, {\em On the reconstruction of a function on a circular domain
      from a sampling of its line integrals}, Journal of Mathematical Analysis and
      Applications, 45 (1974), pp.~357--374.
    
    \bibitem{Matus93finitert}
    {\sc F.~{Matus} and J.~{Flusser}}, {\em Image representation via a finite
      {R}adon transform}, IEEE Transactions on Pattern Analysis and Machine
      Intelligence, 15 (1993), pp.~996--1006.
    
    \bibitem{Midgley2003}
    {\sc P.~Midgley and M.~Weyland}, {\em 3{D} electron microscopy in the physical
      sciences: the development of {Z}-contrast and {EFTEM} tomography},
      Ultramicroscopy, 96 (2003), pp.~413--431.
    \newblock Proceedings of the International Workshop on Strategies and Advances
      in Atomic Level Spectroscopy and Analysis.
    
    \bibitem{Natterer}
    {\sc F.~Natterer}, {\em The Mathematics of Computerized Tomography}, Society
      for Industrial and Applied Mathematics, 2001.
    
    \bibitem{Press06drt}
    {\sc W.~H. Press}, {\em Discrete {R}adon transform has an exact, fast inverse
      and generalizes to operations other than sums along lines}, Proceedings of
      the National Academy of Sciences, 103 (2006), pp.~19249--19254.
    
    \bibitem{NRbook92}
    {\sc W.~H. Press, S.~A. Teukolsky, W.~T. Vetterling, and B.~P. Flannery}, {\em
      Numerical Recipes in C (2nd Ed.): The Art of Scientific Computing}, Cambridge
      University Press, USA, 1992.
    
    \bibitem{Quinto83}
    {\sc E.~T. Quinto}, {\em Singular value decompositions and inversion methods
      for the exterior radon transform and a spherical transform}, Journal of
      Mathematical Analysis and Applications, 95 (1983), pp.~437--448.
    
    \bibitem{Radon1917}
    {\sc J.~Radon}, {\em On the determination of functions from their integral
      values along certain manifolds}, IEEE Transactions on Medical Imaging, 5
      (1986), pp.~170--176.
    
    \bibitem{Railo2020}
    {\sc J.~Railo}, {\em Fourier {A}nalysis of {P}eriodic {R}adon {T}ransforms},
      Journal of Fourier Analysis and Applications, 26 (2020), p.~64.
    
    \bibitem{Rim18split}
    {\sc D.~Rim}, {\em {Dimensional Splitting of Hyperbolic Partial Differential
      Equations Using the Radon Transform}}, SIAM Journal on Scientific Computing,
      40 (2018), pp.~A4184--A4207.
    
    \bibitem{Rim20iadrt}
    {\sc D.~Rim}, {\em Exact and fast inversion of the approximate discrete {R}adon
      transform from partial data}, Applied Mathematics Letters, 102 (2020),
      p.~106159.
    
    \bibitem{Rim18mr}
    {\sc D.~Rim and K.~Mandli}, {\em {Displacement Interpolation Using Monotone
      Rearrangement}}, SIAM/ASA Journal on Uncertainty Quantification, 6 (2018),
      pp.~1503--1531.
    
    \bibitem{Hsung96periodicrt}
    {\sc {TaiChiu Hsung}, D.~P.~K. {Lun}, and {Wan-Chi Siu}}, {\em The discrete
      periodic {R}adon transform}, IEEE Transactions on Signal Processing, 44
      (1996), pp.~2651--2657.
    
    \end{thebibliography}

\end{document}